\documentclass{amsart}

\usepackage{amsthm,amssymb,latexsym,amsmath}
\usepackage[all]{xypic}

\newcommand{\calc}{{\mathcal C}}

\newcommand{\catA}[1]{{\mathfrak A}}
\newcommand{\catI}[1]{{\mathfrak I}}
\newcommand{\catS}[1]{{\mathfrak S}}

\newcommand{\C}{\mathbb{C}}

\newcommand{\Z}{\mathbb{Z}}

\newcommand{\coker}{{\rm coker}}
\newcommand{\Hilb}{{\rm Hilb}}

\newcommand{\im}{{\rm Im}~}

\newcommand{\p}[1]{{\mathbb{P}^{#1}}}

\newcommand{\pn}{{\mathbb{P}^n}}

\newcommand{\yy}{\mathbb{Y}}

\DeclareMathOperator{\End}{End}
\DeclareMathOperator{\ext}{Ext}
\DeclareMathOperator{\Hom}{Hom}
\DeclareMathOperator{\ho}{H}

\DeclareMathOperator{\Mat}{Mat}

\newtheorem{theorem}{Theorem}[section]

\newtheorem{proposition}[theorem]{Proposition}
\newtheorem{lemma}[theorem]{Lemma}
\newtheorem{corollary}[theorem]{Corollary}

\newtheorem{definition}[theorem]{{\bf Definition}}

\begin{document}

\title[Commuting matrices and the Hilbert scheme of points]{Commuting matrices and the Hilbert scheme of points on affine spaces}

\author{Abdelmoubine A. Henni}
\author{Marcos Jardim}

\address{Universidade Federal de Santa Catarina \\
Departamento de Matem\'atica \\
Campus Universit\'ario Trindade\\
CEP 88.040-900 Florian\'apolis-SC, Brasil} 
\email{henni.amar@ufsc.br}

\address{IMECC - UNICAMP \\
Departamento de Matem\'atica \\
Rua S\'ergio Buarque de Holanda, 651 \\
Cidade Universit\'aria \\
13083-859 Campinas--SP, Brazil}
\email{jardim@ime.unicamp.br}

\begin{abstract}
We give linear algebraic and monadic descriptions of the Hilbert scheme of points on the affine space of dimension $n$ which naturally extends Nakajima's representation of the Hilbert scheme of points on the plane. As an application of our ideas and recent results from the literature on commuting matrices, we show that the Hilbert scheme of $c$ points on $(\C^3)$ is irreducible for $c\le 10$.
\end{abstract}

\maketitle
\tableofcontents

%----------------------------------------------------------------------------------------------------
%----------------------------------------------------------------------------------------------------

\section{Introduction}

The Hilbert scheme $\Hilb^{[c]}({\C^{n}})$ of $c$ points in the affine space of dimension $n$ parametrizes $0$-dimensional subschemes of $\C^{n}$ of length $c$. The case of $n=2$ is much studied: $\Hilb^{[c]}({\C^{2}})$ is an irreducible, nonsingular quasi-projective variety of dimension $2c$; in addition, it admits a hyperk\"ahler structure.
The structure of $\Hilb^{[c]}({\C^{n}})$ for $n\ge 3$ is much less understood: $\Hilb^{[c]}({\C^{3}})$ irreducible for $c\le8$ \cite{CEVV}, while it is reducible for $c\ge 78$ \cite{Iar85}; for $n\ge 4$, $\Hilb^{[c]}({\mathbb{C}^{n}})$ is irreducible if and only if $c\le 7$ \cite{CEVV}; however, $\Hilb^{[c]}({\C^{n}})$ is always connected \cite{H}.

The linear algebraic and monadic descriptions of $\Hilb^{[c]}({\C^{2}})$ given by Nakajima in \cite[Chapters 1 \& 2]{N2} are particularly relevant to us. In this paper, we give analogous descriptions of $\Hilb^{[c]}(\C^n)$ and of $\Hilb^{[c]}(\mathbb{Y})$, the Hilbert scheme of $c$ points in an affine variety $\mathbb{Y}$, naturally extending Nakajima's description of $\Hilb^{[c]}(\C^2)$. This goal is attained in the first part of the paper, see Sections \ref{Matrix-Para} through \ref{Rep-mod-functor}, and Section \ref{hilb_y} for the case of affine varieties. 

More precisely, let $V$ and $W$ be complex vector spaces of dimension $c$ and $1$, respectively. Let $B_0,\dots,B_{n-1}$ be operators on $V$ commuting with each other and consider a map $I:W\to V$. The $(n+1)$-tuple $(B_0,\dots,B_{n-1},I)$ is said to be \emph{stable} if there is no proper subspace $S\subset V$ which is invariant under each operator $B_k$ and contains the image of $I$. The group $GL(V)$ acts on the set of all such $(n+1)$-tuple by change of basis on $V$.

We prove that there is a one-to-one corresponding between the following objects:
\begin{enumerate}
\item ideals $J$ in the ring of polynomials $\C[x_{0},\dots,x_{n-1}]$ whose quotient has dimension $c$;
\item stable $(n+1)$-tuple $(B_0,\dots,B_{n-1},I)$ with $\dim V=c$, modulo the action of $GL(V)$;
\item complexes of the form, called \emph{perfect extended monads}:
\begin{equation*}
{\small \xymatrix@C-1.2pc{0\ar[r]&
V_{1-n}\otimes\mathcal{O}_{\pn}(1-n)\ar[r]^{\alpha_{1-n}} & V_{2-n}\otimes\mathcal{O}_{\pn}(2-n) &\hdots\ar[r]^{\alpha_{-1}\hspace{0.5cm}}& V_{0}\otimes\mathcal{O}_{\pn}\ar[r]^{\alpha_{0}\hspace{0.2cm}} & V_{1}\otimes\mathcal{O}_{\pn}(1)\ar[r]&0
}}
\end{equation*}
where $V_1:=V$, $V_0=V^{\oplus n}\oplus\mathbb{C}$ and $V_{i}=V^{\oplus \binom{n}{1-i}}$ for $i<0$,
which are exact everywhere except at degree $0$ (grading of the complex is given by the twisting).
\end{enumerate}
Furthermore, using the above correspondence, we also show that the $\Hilb^{[c]}(\C^n)$ is isomorphic (as a scheme) to a GIT quotient of $\calc(n,c)\times\Hom(W,V)$ by $GL(V)$, where $\calc(n,c)$ denotes the variety of $n$ commuting $c\times c$ matrices.

The correspondence between items (1) and (2) as well as the isomorphism between Hilbert scheme and the GIT quotient of $\calc(n,c)\times\Hom(W,V)$ by $GL(V)$ are already present in the representation theory literature, see for instance \cite{See}, \cite[Appendix by M. V. Nori]{Seshadri} and \cite{Van}, and more recently \cite{Vacca2,Gustavsen,Vacca1}. %However, our presentation is more down-to-earth and closer to Nakajima's description which is familiar to algebraic geometers. 

The correspondence with the so-called perfect extended monads is new. In fact, we introduce in Section \ref{ext-monads} a new class of objects, \emph{extended monads} (cf. Definition \ref{l-ext}), which generalize the usual monads originally introduced by Horrocks in the 1960's \cite{HO} and much studied by several authors since then. Furthermore, Fl\o ystad showed that the dual of the Horrocks--Mumford rank 2 bundle on $\p4$ also arises as the cohomology of an extended monad cf. \cite{Floystad}, while extended monads also arised in the mathematical physics literature
\cite{Szabo}. With this motivations in mind, the basic theory of extended monads is developed here, with a focus on what we call perfect extended monads. We provide a cohomological characterization of the sheaves that arise as cohomology of a perfect extended monad on projective spaces (see Proposition \ref{Character} below), showing, in particular, that ideal sheaves of zero dimensional subschemes do satisfy the required conditions.

We complement our discussion on the parametrization via linear algebra of the Hilbert scheme of points on affine varieties by providing a similar, but more general, description of the Hilbert scheme 
$\Hilb^{[c]}(\yy)$ of $c$ points on an affine variety $\yy\subset\C^n$. More precisely, suppose that $\yy$ is given by algebraic equations $f_1=\cdots=f_l=0$; we show in Section \ref{hilb_y} that a point in $\Hilb^{[c]}(\yy)$ corresponds to a stable $(n+1)$-tuple $(B_0,\dots,B_{n-1},I)$ such that $f_k(B_0,\dots,B_{n-1})=0$ for each $k=1,\dots,l$. Moreover, such correspondence yields a schematic isomorphism between $\Hilb^{[c]}(\yy)$ and the variety of commuting matrices satisfying $f_k(B_0,\dots,B_{n-1})=0$ plus a vector, modulo $GL(V)$.

Finally, as an application of our ideas, it is not difficult to see that $\Hilb^{[c]}(\C^n)$ is irreducible whenever ${\mathcal C}(n,c)$ is irreducible, see details in Section \ref{irred} below. It then follows from recent results due to \u{S}ivic \cite[Theorems 26 \& 32]{S}, that $\Hilb^{[c]}(\C^3)$ is irreducible if $c\le 10$, while this was known to be the case only for $c\le8$, cf. \cite[Theorem 1.1]{CEVV}.

\bigskip

\paragraph{\bf Acknowledgments.}
We would like to thank D. Erman for useful discussions and suggestions about the irreducibility results of Hilbert schemes of points.

AAH was supported by the FAPESP post-doctoral grant number 2009/12576-9. MJ is partially supported by the CNPq grant number 303332/2014-0 and the FAPESP grant number 2014/14743-8.

%%%%%%%%%%%%%%%%%%%%%%%%%%%%%%%%%%%%%%%%%%%%%%%%%%%%%%%%%%%%%%%%%%%%%%%%%%%%%%%%%%%%%%%%%%%%%%%%%

\section{Commuting matrices and stable ADHM data}\label{Matrix-Para}

In this section we shall introduce the necessary material to our construction: let $V$ be a complex vector space of dimension $c$ and let $B_{0},B_{1},\ldots,B_{n-1}\in\End(V)$  be $n$ linear operators on $V$.

\begin{definition}
The variety $\mathcal{C}(n,c)$ of $n$ commuting linear operators on $V$ is the subvariety of $\End(V)^{\oplus n}$ given by
$$ \mathcal{C}(n,c)=\left\{(B_{0},B_{1},\ldots,B_{n-1})\in\End(V)^{\oplus n} ~|~  [B_i, B_j]=0, \, \forall i\neq j\right\} . $$
\end{definition}

The commutation relations can be thought of as a system of ${n\choose2}c^{2}$ homogeneous equations of degree $2$ in $nc^{2}$ variables. 

Let $W$ be a $1$-dimensional complex vector space; one can form the space
$$ \mathbb{B}:=\End(V)^{\oplus n}\oplus \Hom(W,V) $$
whose points are represented by the $(n+1)$-tuple $X=(B_{0},B_{1},\ldots,B_{n-1},I)$ that will be called an \emph{ ADHM datum}. We then define the \emph{variety of ADHM data} $\mathcal{V}(n,c)$ as the subvariety of $\mathbb{B}$ given by
$$ \mathcal{V}(n,c):= \mathcal{C}(n,c)\times\Hom(W,V). $$

\begin{definition}\label{stability}
An ADHM datum $X=(B_{0},B_{1},\ldots,B_{n-1},I)\in\mathbb{B}$ is said to be \emph{stable} if there is no proper subspace $S \subsetneq V$ such that 
$$ B_{0}(S),B_{1}(S),\cdots,B_{n-1}(S), I(W)\subset S. $$
\end{definition}

The set of stable points in $\mathbb{B}$ will be denoted by $\mathbb{B}^{st}$; $\mathcal{V}(n,c)^{st}:= \mathbb{B}^{st}\cap\mathcal{V}(n,c)$ will denote the set of stable points in $\mathcal{V}(n,c).$

\bigskip

Next, we introduce the action of the linear group $G:=GL(V)$ on $\mathbb{B}$. For all $g\in G$ and $X=(B_{0},\ldots,B_{n-1},I)\in \mathbb{B},$ this action is given by
$$ g\cdot(B_{0},\ldots,B_{n-1},I) = (gB_{0}g^{-1},\ldots,gB_{n-1}g^{-1},gI). $$
For a fixed ADHM datum $X,$ we will denote by $G_X$ its stabilizer subgroup:
$$ G_X:=\{g \in G \,|\, gX = X\}\subseteq G. $$
It is easy to see that $X$ is stable if and only if $gX$ is stable, and that $G$ acts on $\mathcal{V}(n,c)$.

We conclude this section with two results relating stability in the sense of Definition \ref{stability} with GIT stability, following the construction in \cite[\S2]{K}.

\begin{proposition}\label{prop4}
If $X\in\mathcal{V}(n,c)^{st}$, then its stabilizer subgroup $G_{X}$ is trivial.
\end{proposition}

\begin{proof}
Let $X=(B_{0},\ldots,B_{n-0},I)$ be a stable ADHM datum and suppose that there exists an element $g\neq\mathbf 1$ in $G$ such that $gI=I$ and $gB_{i}g^{-1}=B_{i}$ for all $i\in\{0, \ldots,n-1\}.$ Then  $\ker(g-\mathbf1)$ is $B_{i}$-invariant, for all $i\in\{0, \ldots,n-1\},$ and $\im I\subseteq\ker(g-\mathbf 1).$ Since $X$ is stable, then $\ker(g-\mathbf 1)\subset V$ must be equal to $V.$ Hence $g$ must be the identity.
\end{proof}

Let $\Gamma(\mathcal{V}(n,c))$ be the ring of regular functions on $\mathcal{V}(n,c).$ Fix $l>0,$ and consider the group homomorphism $\chi:G\to\mathbb{C}^{\ast}$ given by $\chi(g)=(\det g)^{l}.$ This can be used for the of construction a suitable linearization of the $G$-action on $\mathcal{V}(n,c),$ that is, to lift the action of
$G$ on $\mathcal{V}(n,c)$ to an action on $\mathcal{V}(n,c)\times\mathbb{C}$ as follows:
$g\cdot(X,z):=(g\cdot X,\chi(g)^{-1}z)$ for any ADHM datum $X\in\mathcal{V}(n,c)$ and $z\in\mathbb{C}.$ Then one can form the scheme
$$\mathcal{V}(n,c) /\!/_{\chi}G :=
{\rm Proj}\left(\bigoplus_{i\geq0}\Gamma(\mathcal{V}(n,c))^{G,\chi^{i}} \right) $$
where
$$\Gamma(\mathcal{V}(n,c))^{G,\chi^{i}}:=\left\{ f\in\Gamma(\mathcal{V}(n,c)) ~|~
f(g\cdot X)=\chi(g)^{-1}\cdot f(X),\hspace{0.2cm}\forall g\in G\right\}. $$
The scheme $\mathcal{V}(n,c)/\!/_{\chi}G$ is projective over the ring $\Gamma(\mathcal{V}(n,c))^{G}$ and quasi-projective over
$\mathbb{C}.$

\begin{proposition}\label{closed orbit}
The orbit $G\cdot(X,z)$ is closed, for $z\neq0$, if and only if  the
ADHM datum $X\in\mathcal{V}(n,c)$ is a stable.
\end{proposition}

\begin{proof}
The proof is similar to \cite[Proposition 2.10]{HJM}
\end{proof}

From Propositions \ref{prop4} and \ref{closed orbit} and since the group $G$ is reductive, it follows that the quotient space
$\mathcal{M}(n,c):=\mathcal{V}(n,c) /\!/_{\chi}G$ is a good categorical quotient \cite[Thm. 1.10]{mumford}. Furthermore, GIT tells us that the GIT quotient $\mathcal{M}(n,c)$ is the space of orbits $G\cdot X\subset \mathcal{V}(n,c)$ such that the lifted
orbit $G\cdot(X,z)$ is closed within $\mathcal{V}(n,c)\times\mathbb{C}$ for all $z\neq0.$ We conclude therefore, from Proposition \ref{closed orbit}, that 
$$ \mathcal{M}(n,c) = \mathcal{V}(n,c)^{\rm st} / G . $$

%%%%%%%%%%%%%%%%%%%%%%%%%%%%%%%%%%%%%%%%%%%%%%%%%%%%%%%%%%%%%%%%%%%%%%%%%%%%%%%%%%%%%%%%%%%%%%%%%%%%%%%%%%%%%%%%%%%%%%%%%%%%

Finally we end this section by mentioning the following results: As a set, the Hilbert scheme of $c$ points on $\mathbb{C}^{n}$ is given by:
$$ {\rm Hilb}^{[c]}(\mathbb{C}^n) = \{ I \triangleleft \mathbb{C}[z_0,\dots,z_{n-1}] ~|~ 
\dim_\mathbb{C}(\mathbb{C}[z_0,\dots,z_{n-1}]/I)=c \} . $$

The existence of its schematic structure is a special case of the general result of Grothendieck \cite{Groth1}. Another explicit construction of the Hilbert scheme of points on the affine plane is given by Nakajima \cite{N2}. The reader may also consult \cite{Nitsure} for more general results and examples.

\bigskip

We conclude this section by stating the following:

\begin{theorem}\label{Corresp}
There exists a set-theoretical bijection between the quotient space $\mathcal{M}(n,c)$ and the Hilbert scheme of $c$ points in $\mathbb{C}^{n}.$
\end{theorem}

This is achieved by constructing, for every stable ADHM datum
$$ X=(B_{0},\ldots,B_{n-1},I)\in\mathcal{V}(n,c)^{\rm st}, $$
the surjective linear transformation:
\begin{center}
$\begin{array}{cccl} \Phi_{X}: & \mathbb{C}\left[Z_{0},\ldots,Z_{n-1}\right] & \longrightarrow & \qquad V \\ & p(Z_{0},\ldots,Z_{n-1}) & \longmapsto &  p(B_{0},\ldots,B_{n-1})I(1) \end{array}$
\end{center} In particular, $\mathbb{C}\left[Z_{0},\ldots,Z_{n-1}\right]/\ker \Phi_{X}$ is isomorphic to $V$.
Clearly this maps descend to a map \begin{center}
$\begin{array}{cccl}\Psi : &  \mathcal{M}(n,c) & \longrightarrow & \Hilb^{[c]}({\mathbb{C}^{n}}) \\ & [X] & \longmapsto & \ker \Phi_{X}\end{array},$
\end{center}
and showing that it is a bijection is an exercise left to the reader.

%%%%%%%%%%%%%%%%%%%%%%%%%%%%%%%%%%%%%%%%%%%%%%%%%%%%%%%%%%%%%%%%%%%%%%%%%%%%%%%%%%%%%%%%%%%%%%%%%%%%%%

\section{Extended monads and perfect extended monads}\label{ext-monads}

In this section we shall generalize the concept of \emph{monads}, introduced by Horrocks (the reader may consult \cite{OSS} for definitions and properties), in order to describe ideal sheaves for zero-dimensional subschemes of $\C^{n}$ and $\pn$, $n\geq2.$

Let $X$ be a smooth projective algebraic variety of dimension $n$ over the field of complex numbers $\C$, and let $\mathcal{O}_{X}(1)$ be a polarization on it.

%%%%%%%%%%%%%%%%%%%%%%%%%%%%%%%%%%%%%%%%%%%%%%

\subsection{$l$-extended monads}

We introduce the following generalization of monads.

\begin{definition}\label{l-ext}
An \emph{$l$-extended monad} over $X$ is a complex
\begin{equation}\label{premordial}
C^{\bullet}:\quad  0\to C^{-l-1}\stackrel{\alpha_{-l-1}}{\longrightarrow} C^{-l}\stackrel{\alpha_{-l}}{\longrightarrow}\cdots\stackrel{\alpha_{-2}}{\longrightarrow} C^{-1}\stackrel{\alpha_{-1}}{\longrightarrow} C^{0}\stackrel{\alpha_{0}}{\longrightarrow} C^{1}\to0
\end{equation}
of locally free sheaves over $X$ which is exact at all but the $0$-th position, i.e. $\mathcal{H}^{i}(C^{\bullet})=\ker\alpha_i/\im\alpha_{i-1}=0$ for $i\ne0$. 
The coherent sheaf $\mathcal{E}:=\mathcal{H}^{0}(C^{\bullet})=\ker \alpha_{0}/\im \alpha_{-1}$ will be called \emph{the cohomology of $C^{\bullet}$}.
 \end{definition}
 
Note that a monad on $X$, in the usual sense, is just a $0$-extended monad.

Moreover, one can associate to any $l$-extended monad $C^{\bullet}$ a \emph{display} of exact sequences as the following
\begin{equation}\label{ext-display}
\xymatrix@R-0.8pc@C0.5pc@u{& 0\ar[d] & 0\ar[d] &  & \\
& C^{-l-1}\ar[d]_{\alpha_{-l-1}}\ar@{=}[r] & C^{-l-1}~~~~~~\ar[d]^{\alpha_{-l-1}} & & \\
& C^{-l}\ar@{=}[r]\ar[d]_{\alpha_{-l}} & C^{-l}\ar[d]^{\alpha_{-l}} & & \\
& \vdots\ar[d]_{\alpha_{-2}} & \vdots\ar[d]^{\alpha_{-2}} & & \\
& C^{-1}\ar[d]\ar@{=}[r] & C^{-1}\ar[d]^{\alpha_{-1}} & & \\
0\ar[r]& K\ar[d]\ar[r] & C^{0}\ar[d]\ar[r]^{\alpha_{0}} & C^{1}\ar@{=}[d]\ar[r] &0 \\
0\ar[r]& \mathcal{E}\ar[d]\ar[r] & Q\ar[d]\ar[r] & C^{1}\ar[r] & 0\\
& 0& 0 &  &
}
\end{equation}
where $K:=\ker\alpha_{0}$ and $Q:=\coker\alpha_{-1}$.

A morphism $\phi:C_{1}^{\bullet}\to C_{2}^{\bullet}$ between two $l$-extended monads $C_{1}^{\bullet}$ and $C_{2}^{\bullet}$ is an $(l+3)$-tuple of morphisms such that the following diagram commutes:
\begin{equation}\label{morph}
\xymatrix@C-0.5pc@R-1,2pc{
C_{1}^{\bullet}:\ar[d]^{\phi} &0\ar[r]& C_{1}^{-l-1}~~~~~~\ar[d]_{\phi_{-l-1}}\ar[r]^{\alpha^{1}_{-l-1}}&C_{1}^{-l}\ar[d]^{\phi_{-l}}&\cdots\ar[r]^{\alpha^{1}_{-2}}&C_{1}^{-1}\ar[d]^{\phi_{-1}}\ar[r]^{\alpha^{1}_{-1}} & C_{1}^{0}\ar[d]^{\phi_{0}}\ar[r]^{\alpha^{1}_{0}}&C_{1}^{1}\ar[d]^{\phi_{1}}\ar[r]&0 \\
C_{2}^{\bullet}:&0\ar[r]& C_{2}^{-l-1}~~~~~~\ar[r]^{\alpha^{2}_{-l-1}}&C_{2}^{-l}&\cdots\ar[r]^{\alpha^{2}_{-2}}&C_{1}^{-1}\ar[r]^{\alpha^{2}_{-1}} & C_{2}^{0}\ar[r]^{\alpha^{2}_{0}}&C_{2}^{1}\ar[r]&0
}
\end{equation}

With these definitions, the category of $l$-extended monads form a full subcategory of the category $Kom^{\flat}(X)$ of bounded
complexes of coherent sheaves on $X.$ 

$l$-extended monads have already appeared in the literature. The most important example of a locally-free sheaf that can be obtained as the cohomology of a 2-extended monad on $\p4$ is the dual of the Horrocks--Mumford bundle; indeed, Fl\o ystad shows in \cite[Introduction: example b.]{Floystad} that the Horrocks--Mumford bundle is given by the cohomology at degree zero, where the grading is given by the twist, of a complex of the form
$$ 0\to\mathcal{O}^{\oplus 5}_{\mathbb{P}^{4}}(-1) \to \mathcal{O}^{\oplus 15}_{\mathbb{P}^{4}} \to
\mathcal{O}^{\oplus 10}_{\mathbb{P}^{4}}(1) \to \mathcal{O}^{\oplus 2}_{\mathbb{P}^{4}}(2)\to0. $$
Dualizing such complex we get a 2-extended monad on $\p4$ whose cohomology is the dual of the
Horrocks--Mumford bundle.

Moreover, objects very closely related to 2-extended monads on $\p3$ have also appeared in the mathematical physics literature, see \cite[Section 4]{Szabo}.

An $l$-extended monad can be broken into the following two complexes: first,
\begin{equation}\label{Resol}
\xymatrix@C-0.5pc{ N^{\bullet}:&0\ar[r]&C^{-l-1}~~~~~~\ar[r]^{\alpha_{-l-1}}&C^{-l}&\cdots\ar[r]^{\alpha_{-3}}&C^{-2}\ar[r]^{\alpha_{-2}} & C^{-1}\ar[r]^{J_{-1}}&\mathcal{G}\ar[r]&0
} \end{equation}
which is exact, and a locally free resolution of the sheaf $\mathcal{G}=\coker\alpha_{-2},$ and
\begin{equation}\label{Monad-like}
\xymatrix@C-0.5pc{ M^{\bullet}:& \mathcal{G}\ar[r]^{I_{-1}} & C^{0}\ar[r]^{\alpha_{0}}&C^{1}
} \end{equation}
where $I_{-1}\circ J_{-1}=\alpha_{-1}$. $M^{\bullet}$ is a monad-like complex in which the coherent sheaf $\mathcal{G}$ might not be locally free; indeed, $\mathcal{G}$ is not locally free for the extended monads describing ideal sheaves of $0$-dimensional subschemes, the situation most relevant to the present paper. 

For a given $l$-extended monad, we refer to the complexes $M^{\bullet}$ and $N^{\bullet}$ as the \emph{associated resolution} and the \emph{associated monad}, respectively. Therefore, the morphism $\phi:C_{1}^{\bullet}\to C_{2}^{\bullet}$ can be thought of as a pair of morphisms $(\phi_{N}:N_{1}^{\bullet}\to N_{2}^{\bullet},\phi_{M}:M_{1}^{\bullet}\to M_{2}^{\bullet})\in\Hom(N_{1}^{\bullet},N_{2}^{\bullet})\times\Hom(M_{1}^{\bullet},M_{1}^{\bullet}).$

Remark that as long as we have $\phi_{0}(\im\alpha_{1}^{-1})\subseteq\im\alpha_{2}^{-1}$ and $\phi_{0}(\ker\alpha_{0}^{1})\subseteq\ker\alpha_{0}^{2}$ then $\phi$ is determined by only $\phi_{0};$ indeed, the conditions $$\phi_{0}(\im\alpha_{1}^{-1})\subseteq\im\alpha_{2}^{-1}\quad\textnormal{ and }\quad\phi_{0}(\im I_{1}^{-1})\subseteq\im I_{2}^{-1}$$ are equivalent (here we considered the morphism of the associated monads). Hence $\phi_{0}$ determines the morphism $\phi_{M},$ and consequently it also determines the whole morphism $\phi:C_{1}^{\bullet}\to C_{2}^{\bullet}.$ This is because $N_{1}^{\bullet}$ and $N^{\bullet}_{2}$ are locally free resolutions and hence projective resolutions, so that giving a morphism $\phi_{G}:\mathcal{G}_{1}\to\mathcal{G}_{2}$ determines all the morphisms $\phi_{-i}:C_{1}^{-i}\to C_{2}^{-i}$ for $i\leq-1.$

Since taking cohomology is a functorial operation, a morphism
$\phi:C_{1}^{\bullet}\to C_{2}^{\bullet}$ of two $l$-extended monads
$C_{1}^{\bullet}$ and $C_{2}^{\bullet},$ induces a morphism between
their respective cohomologies
$$H(\phi):\mathcal{H}^{0}(C_{1}^{\bullet})\to\mathcal{H}^{0}(C_{2}^{\bullet}).$$
Of course, isomorphic complexes induce isomorphic cohomologies. It
follows that there is natural notion of equivalence for $l$-extended
monads with the same terms $C^{i}$ provided by the action of the
automorphism group
$\mathcal{A}ut(C^{\bullet})=\mathcal{A}ut(C^{-l-1})\times\mathcal{A}ut(C^{-l})\times\cdots\times\mathcal{A}ut(C^{0})\times\mathcal{A}ut(C^{1}).$

Our goal now is to study families of ideal sheaves of zero-cycles in $\pn$. It turns out that such ideal sheaves are given by cohomologies of a special kind of $l$-extended monads. However, before proving this claim, which will be done only in Section \ref{z-cy} below, we tackle a more general question, namely under which conditions a homomorphism $\mathcal{H}^{0}(C_{1}^{\bullet})\to\mathcal{H}^{0}(C_{2}^{\bullet})$  lifts to a homomorphism $C_{1}^{\bullet}\to C_{2}^{\bullet}$ between the corresponding complexes. In particular to determine the automorphisms of such objects.   

Our next result provides a sufficient condition, by showing when the cohomology functor is full and faithful.

\begin{proposition}\label{full-faith}
Let
$$\xymatrix@C-0.8pc{
C_{1}^{\bullet}:& 0\ar[r]&C_{1}^{-l-1}~~~~~~\ar[r]^{\alpha^{1}_{-l-1}}&C_{1}^{-l}&\cdots\ar[r]^{\alpha^{1}_{-2}}&C_{1}^{-1}\ar[r]^{\alpha^{1}_{-1}} & C_{1}^{0}\ar[r]^{\alpha^{1}_{0}}&C_{1}^{1}\ar[r]& 0&
\textnormal{ and} }$$
$$\xymatrix@C-0.8pc{
C_{2}^{\bullet}:& 0\ar[r]&C_{2}^{-l-1}~~~~~~\ar[r]^{\alpha^{2}_{-l-1}}&C_{2}^{2-n}&\cdots\ar[r]^{\alpha^{2}_{-2}}&C_{2}^{-1}\ar[r]^{\alpha^{2}_{-1}} & C_{2}^{0}\ar[r]^{\alpha^{2}_{0}}&C_{2}^{1}\ar[r]&0 
}$$ 
be two $l$-extended monads, and let $\mathcal{E}_{1}$ and $\mathcal{E}_{2}$ denote their respective cohomologies. Then
$$ H : \Hom(C^{\bullet}_{1},C^{\bullet}_{2})\to\Hom(\mathcal{E}_{1},\mathcal{E}_{2}) $$
is surjective if 
$$\ext^{1}(C^{1}_{1},C^{0}_{1})=0,$$ 
$$\ext^{k}(C^{0}_{1},C^{-k}_{2})=0 \textnormal{ for } k\geq1,\quad\quad
\ext^{k}(C^{1}_{1},C^{-k+1}_{2})=0 \textnormal{ for } k\geq2. $$
Moreover if
$$\Hom(C^{1}_{1},C^{0}_{2})=0,$$
$$\ext^{k}(C^{0}_{1},C^{-k+1}_{2})=0 \textnormal{ for } k\geq1,\quad\quad
\ext^{k}(C^{0}_{1},C^{-k-1}_{2})=0 \textnormal{ for all } k\geq0,$$
then $H$ is an isomorphism.
\end{proposition}

\begin{proof}
Let $\mathcal{G}_{1}=\im\alpha^{1}_{-1}$ and $\mathcal{G}_{2}=\coker\alpha^{2}_{-1}.$ The associated resolution $N^{\bullet}_{2}$ can be broken into sequences
$$0\to\mathcal{G}^{-i}_{2}\to C^{-i}_{2}\to\mathcal{G}^{-i+1}_{2}\to0,\quad 1\leq i\leq l+1$$ where we put $\mathcal{G}^{-l-1}_{2}=C^{-l}_{2}$ and $\mathcal{G}^{0}_{2}=\mathcal{G}_{2}.$ Then, by applying either $\Hom(C^{0}_{1},\bullet)$ or $\Hom(C^{1}_{1},\bullet)$ on the above sequences and incorporating the conditions given in the proposition, it follows that $H$ is surjective if $$\ext^{1}(C^{1}_{1},C^{0}_{1})=\ext^{1}(C^{0}_{1},\mathcal{G}_{2})=\ext^{2}(C^{1}_{1},\mathcal{G}_{2})=0,$$ and it is an isomorphism if $$\Hom(C^{1}_{1},C^{0}_{2})=\Hom(C^{0}_{1},\mathcal{G}_{2})=\ext^{1}(C^{0}_{1},\mathcal{G}_{2})=0.$$ To finish the proof, it suffice to apply \cite[Chapter II, Lemma 4.1.3]{OSS} to the associated monad $M^{\bullet}_{2}$ of the $l$-extended monad $C^{\bullet}_{2}.$
\end{proof}

In particular, one obtains the following statement.

\begin{corollary}
In the notation of Proposition \ref{full-faith}, assume that
$$\ext^{1}(C^{1}_1,C^{0}_1)=0,$$ 
$$\ext^{k}(C^{0}_1,C^{-k}_2)=0 \textnormal{ for } k\geq1,\quad\quad
\ext^{k}(C^{1}_1,C^{-k+1}_2)=0 \textnormal{ for } k\geq2.$$
Then $\mathcal{E}_1$ and $\mathcal{E}_2$ are isomorphic if and only if $C^{\bullet}_1$ and $C^{\bullet}_2$ are isomorphic (as $l$-extended monads).
\end{corollary}

%%%%%%%%%%%%%%%%%%%%%%%%%%%%%%%%%%%%

\subsection{Perfect extended monads}

We now introduce the class of $l$-extended monads which is relevant to the description of the Hilbert scheme of points. Recall that $\mathcal{O}_{X}(1)$ is the chosen polarization on the $n$-dimensional projective algebraic variety $X.$

\begin{definition}
A \emph{perfect extended monad} on a $n$-dimensional projective \linebreak variety $X$ is a $(n-2)$-extended monad $P^{\bullet}$ on $X$ of the following form
$$\xymatrix@C-0.8pc{ 0\ar[r]&\mathcal{O}_{X}(1-n)^{\oplus a_{1-n}}\ar[r]^{\alpha_{1-n}}&
\mathcal{O}_{X}(2-n)^{\oplus a_{2-n}}\ar[r]&}\hspace{1cm}$$
$$\hspace{3cm}\xymatrix@C-0.5pc{\cdots\ar[r]^{\alpha_{-2}\hspace{0.5cm}}&\mathcal{O}_{X}(-1)^{\oplus a_{-1}}\ar[r]^{\hspace{0.3cm}\alpha_{-1}} &\mathcal{O}_{X}^{\oplus a_{0}}\ar[r]^{\alpha_{0}\hspace{0.1cm}}&\mathcal{O}_{X}(1)^{\oplus a_{1}}\ar[r]&0
},$$ for some integers $a_{i},$ $1-n\leq i\leq 1.$ 
\end{definition}

We recall to the reader that a projective scheme $X$ is \emph{arithmetically Cohen--Macaulay}, or simply \emph{ACM}, if its homogeneous coordinate ring is Cohen--Macaulay ring. Moreover let us denote by $\mathfrak{P}er$ the full subcategory of $Kom^{\flat}(X)$ consisting of perfect extended monads.

\begin{corollary}\label{fully-faithfull}
If $X$ is an $n$-dimensional ACM variety, then the cohomology functor
$$ H:\mathfrak{P}er(X)\to{\rm Coh}(X) $$
is full and faithful.
\end{corollary}

\begin{proof}
This follows easily from Proposition \ref{full-faith}: since $X$ is ACM, we have that $$\Hom(C^{1}_{1}, C^{0}_{2})=\ho^{0}(\mathcal{O}_{X}(-1))=0\quad\textnormal{ and}$$ $$\ext^{i}(\mathcal{O}_{X}(a),\mathcal{O}_{X}(b))=\ho^{i}(\mathcal{O}_{X}(b-a))=0, \textnormal{ for } 1\leq i\leq n-1.$$ 
\end{proof}

It follows from the Corollary above that the automorphism group of a perfect extended monad on an ACM variety is just
$GL_{a_{1-n}}(\mathbb{C})\times GL_{a_{-n}}(\mathbb{C})\times\cdots\times GL_{a_{1}}(\mathbb{C}).$

We finish this section by describing the cohomology of sheaves which are in the image of the functor $H$ on $\pn,$ $n\geq 2$.

\begin{proposition}\label{hom-char}
If $\mathcal{E}$ is the cohomology of a perfect extended monad on $\pn$ ($n\ge 2$) then:
\begin{itemize}
\item[(i)] $\ho^{0}(\mathcal{E}(k))=0$ for $k<0$;
\item[(ii)] $\ho^{n}(\mathcal{E}(k))=0$ for $k>-n-1$;
\item[(iii)] $\ho^{i}(\mathcal{E}(k))=0$ $\forall k,$ $2\leq i\leq n-1$, when $n\ge3$.
\end{itemize}
\end{proposition}

\begin{proof}
We twist the middle column of the display \eqref{ext-display} by $\mathcal{O}_{\pn}(k),$  then break it into short exact sequences
{\small\begin{equation} \label{broken}
\xymatrix@C-0.8pc@R-2.6pc{ 0\ar[r]&\mathcal{O}_{\pn}(k+1-n)^{\oplus a_{1-n}}\ar[r]&\mathcal{O}_{\pn}(k+2-n)^{\oplus a_{2-n}}\ar[r]&Q_{2-n}(k)\ar[r]&0 \\
&&\vdots&& \\
0\ar[r]&Q_{-p-1}(k)\ar[r]&\mathcal{O}_{\pn}(k-p)^{\oplus a_{-p}}\ar[r]&Q_{-p}(k)\ar[r]&0 \\
&&\vdots&& \\
0\ar[r]&Q_{-1}(k)\ar[r]&\mathcal{O}_{\pn}(k)^{\oplus a_{0}}\ar[r]&Q_{0}(k)\ar[r]&0 }
\end{equation}
where $Q_{0}:=Q=\coker\alpha_{-1}.$
}

\noindent\underline{\textbf{Step. 1:}}
From the long sequences in cohomology of the first row above, we have $$\ho^{i}(\mathcal{O}_{\pn}(k+2-n))^{\oplus a_{2-n}} \to \ho^{i}(Q_{2-n}(k)) \to \ho^{i+1}(\mathcal{O}_{\pn}(k+1-n))^{\oplus a_{1-n}}\to\cdots $$ 
Then, from the vanishing properties of line bundles on $\pn$, it follows that

$\ho^{0}(Q_{2-n}(k))=0$ for $k<n-2;$ $\ho^{n}(Q_{2-n}(k))=0$ for $k>-1;$ $\ho^{i}(Q_{2-n}(k))=0$ $\forall k,$ $1\leq i\leq n-1.$

\bigskip

\noindent \underline{\textbf{Step. 2:}}
Using induction on the remaining rows in \eqref{broken} it follows that, for $p>2,$

$\ho^{0}(Q_{p-n}(k))=0$ for $k<n-p;$ $\ho^{n}(Q_{p-n}(k))=0$ for $k>-p-1;$ $\ho^{i}(Q_{p-n}(k))=0$ $\forall k,$ $1\leq i\leq n-1.$

\bigskip

\noindent\underline{\textbf{Step. 3:}}
From the long exact sequence in cohomology of the lower row in \eqref{ext-display} twisted by $\mathcal{O}_{\pn}(k)$ one has
$$ \ho^{i-1}(\mathcal{O}_{\pn}(k+1))^{\oplus a_{1}} \to \ho^{i}(\mathcal{E}(k)) \to \ho^{i}(Q(k))\to\cdots $$
Using the vanishing obtained in Step. $2$ for $Q_{0}=Q,$ the claims of items $(i),$ $(ii),$ $(iii)$ and $(iv)$ follow.

The last item is obtained by dualizing the lower row of \eqref{ext-display}.
\end{proof}

Now let us denote by $\Omega^{-p}_{\pn}$ the bundle of holomorphic $(-p)$-forms on $\pn,$ where $p\leq0$ in our convention.

\begin{proposition}\label{Character}
If a coherent sheaf $\mathcal{E}$ on $\pn$ ($n\ge 2$) satisfies:
\begin{itemize}
\item[(i)] $\ho^{0}(\mathcal{E}(-1))=\ho^{n}(\pn,\mathcal{E}(-n))=0$;
\item[(ii)] $\ho^{q}(\mathcal{E}(k))=0 \quad\forall k,\quad2\leq q\leq n-1$ when $n\ge3$;
\item[(iii)] $\ho^{1}(\mathcal{E}\otimes\Omega^{-p}_{\pn}(-p-1))\neq0$ for $-n\le p\leq0$;
\end{itemize}
then $\mathcal{E}$ is the cohomology of a perfect extended monad.
\end{proposition}

\begin{proof}
Applying Beilinson's theorem \cite[Chapter II, Theorem 3.1.4]{OSS} to the sheaf $\mathcal{E}(-1),$ one gets a spectral sequence with  $E_{1}$-term given by 
$$ E_{1}^{p,q}=\ho^{q}(\mathcal{E}\otimes\Omega_{\pn}^{-p}(-p-1))\otimes\mathcal{O}_{\pn}(p) $$ 
which converges to the graded sheaf associated to a filtration of $\mathcal{E}(-1)$ itself. 

Twist the Euler sequence for the sheaves of differential forms
$$ 0 \to \Omega^p(p) \to \mathcal{O}_{\pn}^N \to \Omega^{p-1}(p) \to 0 \quad,\quad 
N=\binom{n+1}{p} $$
by $\mathcal{E}(k-p)$ and use hypotheses \emph{(i)} and \emph{(ii)} above to conclude, after long but straightforward calculations with the associated long exact sequences of cohomology, that $ E_{1}^{p,q}=0$ for $q\ne1$. 

It follows immediately that the Beilinson spectral sequence degenerates already at the $E_2$-term, i.e. $E_2=E_\infty$. Beilinson's theorem then implies that the complex $E_1^{p,1}$ given by
\begin{equation} \label{cpx-one}
V_n\otimes\mathcal{O}_{\pn}(-n) \to \cdots \to V_1\otimes\mathcal{O}_{\pn}(-1) \to V_0\otimes\mathcal{O}_{\pn},
\end{equation}
with $V_p:=H^1(\mathcal{E}\otimes\Omega_{\pn}^{-p}(-p-1))$, $-n\le p\le 0$, is exact everywhere except at position $p=-1$, and its cohomology at this position is precisely $\mathcal{E}(-1)$.

The third hypothesis implies that none of the vector spaces $V_p$ vanishes. So twisting the complex (\ref{cpx-one}) by $\mathcal{O}_{\pn}(1)$, we obtain a perfect extended monad whose cohomology is exactly $\mathcal{E}$, as desired.
\end{proof}

%%%%%%%%%%%%%%%%%%%%%%%%%%%%%%%%%%%%%%%%%%%%%%%%%%%%%%%%%%%%%%%%%%%%%%%%%%%%%%%%%%%%%%%%%%%%%%%%%%%%%%%%%%%%%%%

\section{Ideal sheaves of zero-dimensional subschemes of $\pn$}\label{z-cy}

We now consider sheaves $\mathcal{E}$ of rank $r$ on $\pn$ fitting in the following short exact sequence
\begin{equation}\label{structure}
0 \to \mathcal{E} \to \mathcal{O}_{\pn}^{\oplus r} \to \mathcal{Q} \to 0,
\end{equation}
where $\mathcal{Q}$ is a pure torsion sheaf of length $c$ supported on a $0$-dimensional subscheme $Z\subset\pn$. 

Note that the Chern character of $\mathcal{E}$ is given by $ch(\mathcal{E})=r-cH^{n}$, and that $\mathcal{E}$ is necessarily torsion free. Such sheaves can also be regarded as points in the Quot scheme $Quot^{P=c}(\mathcal{O}_{\pn}^{\oplus r})$. 

In the case $r=1$, it is clear that $\mathcal{E}$ is the sheaf of ideals in $\mathcal{O}_{\pn}$ associated to the zero-dimensional subscheme $Z$, i.e. $\mathcal{Q}=\mathcal{O}_Z$; in this case, we will then denote $\mathcal{E}$ by $\mathcal{I}_{Z}$.

\begin{proposition}\label{vanishing}
Every sheaf $\mathcal{E}$ on $\pn$ given by sequence \eqref{structure} is the cohomology of a perfect extended monad $P^{\bullet}$ with terms of the form $P^{-i}:=V_{i}\otimes\mathcal{O}_{\pn}(i)$, $i=1-n,\dots,0,1$, where
\begin{equation}\label{eq-lema}
V_{i}:=\ho^{1}(\mathcal{E}\otimes\Omega^{1-i}_{\pn}(-i)) \cong
\ho^{0}(\mathcal{Q}\otimes\Omega^{1-i}_{\pn}(-i)).
\end{equation}
Furthermore, we have the following isomorphisms:
\begin{equation}\label{id1}
V_1 \cong \ho^{0}(\mathcal{Q})
\end{equation}
\begin{equation}\label{id2}
V_{i}\cong  
\left\{\begin{array}{ll} V_{1}^{\oplus n}\oplus\mathbb{C}^{r} & \textnormal{for }i=0 \\ 
V_{1}^{\oplus \binom{n}{1-i}}  &  \textnormal{for }i<0 \end{array}\right.
\end{equation}
\end{proposition}

\begin{proof} 
Conditions \emph{(i)} and \emph{(ii)} in Proposition \ref{Character} follow easily from twisting sequence \eqref{structure} by $\mathcal{O}_{\pn}(k)$ and using that fact that $\mathcal{Q}$ is supported in dimension zero. Next, twist sequence \eqref{structure} by $\Omega^{-p}_{\pn}(-p-1)$ and use Bott's formula to obtain the isomorphisms in \eqref{eq-lema}.

The isomorphisms \eqref{id1} and \eqref{id2} can be proved as follows. First, we have 
$ V_1 := \ho^{1}(\mathcal{E}(-1)) \cong \ho^{0}(\mathcal{Q}(-i))\cong \ho^{0}(\mathcal{Q}),$ for $i=1$
since $\mathcal{Q}$ is supported in dimension zero. 

The space $V_{0}$ fits in the sequence
$$ 0 \to \ho^{0}(\mathcal{Q}\otimes\Omega^{1}_{\pn}) \to \ho^{1}(\mathcal{E}\otimes\Omega^{1}_{\pn}) \to \ho^{1}(\Omega^{1}_{\pn})^{\oplus r} \to 0 $$
obtained from sequence \eqref{structure} twisted by $\Omega^{1}_{\pn}$. On the other hand, we know from the Euler sequence that $\ho^{1}(\Omega^{1}_{\pn})\cong\ho^0(\mathcal{O}_{\pn})$.
Moreover, since $ \ho^{0}(\mathcal{Q}\otimes\Omega^{1}_{\pn}) \cong \ho^{0}(\mathcal{Q})^{\oplus n} \cong V_{1}^{\oplus n}, $ it follows that
$\ho^{1}(\mathcal{E}\otimes\Omega^{1-i}_{\pn})\cong V_{1}^{\oplus n}\oplus\mathbb{C}^{r}.$

Finally, note that 
$ V_{i} = \ho^{1}(\mathcal{E}\otimes\Omega^{1-i}_{\pn}(-i)) \cong \ho^{0}(\mathcal{Q}^{\oplus\binom{n}{-i}}) = V_{1}^{\oplus \binom{n}{1-i}},$ for $i<0.$
\end{proof}

In particular, for the case $r=1$, we have the following Corollary. 

\begin{corollary}\label{perfect-ideal}
For every zero dimensional subscheme $Z\subset\pn$, there exists a perfect extended monad $P^\bullet$ of the form
\begin{equation*}
{\small \xymatrix@C-1.2pc{
0\ar[r]&V_{1-n}\otimes\mathcal{O}_{\pn}(1-n)\ar[r]^{\alpha_{1-n}} & V_{2-n}\otimes\mathcal{O}_{\pn}(2-n) &\hdots\ar[r]^{\alpha_{-1}\hspace{0.5cm}}& V_{0}\otimes\mathcal{O}_{\pn}\ar[r]^{\alpha_{0}\hspace{0.2cm}} & V_{1}\otimes\mathcal{O}_{\pn}(1)\ar[r]&0
}}
\end{equation*}
where $V_1:=\ho^0(\mathcal{O}_Z)$ and 
$$ V_{i} \cong \left\{ \begin{array}{ll} 
V_{1}^{\oplus n}\oplus\mathbb{C} & \textnormal{for } i=0 \\
V_{1}^{\oplus \binom{n}{1-i}}  &  \textnormal{for } i<0
\end{array}\right. , $$
whose cohomology is the ideal sheaf $\mathcal{I}_Z$.
\end{corollary}

%%%%%%%%%%%%%%%%%%%%%%%%%%%%%%%%%%%%%%%%%%%%%%%%%%%%%%%%%%%%%%%%%%%%%%%%%%%%%%%%%%%%%%%%%%%%%%%%%%%%%

\subsection{The $\p3$ case}\label{reduction}

Now, we fix a hyperplane $\wp\subset\p3$. We shall describe how to get linear algebraic data out of the perfect extended monad corresponding to a $0$-dimensional
subscheme $Z\subset\p3\setminus\wp$, as in Corollary \ref{perfect-ideal}.

Let us start by fixing notation; we choose homogeneous coordinates $[z_{0};z_{1};z_{2};z_{3}]$ on $\p3$
in such a way that the hyperplane $\wp$ is given by the equation $z_{3}=0$. We also regard such coordinates as a basis for the space of global sections $\ho^{0}(\mathcal{O}_{\p3}(1))$. 

By Corollary \ref{perfect-ideal}, there is a perfect extended monad $P^{\bullet}$ with cohomology equal to the ideal sheaf $\mathcal{I}_{Z}$. It is given by
\begin{equation}\label{cpx-p3}
{\small \xymatrix@C-1pc{
0\ar[r]&V_{1}\otimes\mathcal{O}_{\p3}(-2)\ar[r]^{\alpha_{-2}} & V_{1}^{\oplus3}\otimes\mathcal{O}_{\p3}(-1)\ar[r]^{\alpha_{-1}\hspace{0.5cm}}& (V_{1}^{\oplus3}\oplus W)\otimes\mathcal{O}_{\p3}\ar[r]^{\alpha_{0}\hspace{0.2cm}} & V_{1}\otimes\mathcal{O}_{\p3}(1)\ar[r]&0
}}
\end{equation}

\noindent where
$\alpha_{-2}\in\Hom(V_{1},V_{1}^{\oplus3})\otimes\ho^{0}(\mathcal{O}_{\p3}(1)),$
$\alpha_{-1}\in\Hom(V_{1}^{\oplus3},V_{1}^{\oplus3}\oplus
W)\otimes\ho^{0}(\mathcal{O}_{\p3}(1))$ and
$\alpha_{0}\in\Hom(V_{1}^{\oplus3}\oplus
W,V_{1})\otimes\ho^{0}(\mathcal{O}_{\p3}(1)).$

 we can write
the $\alpha$'s as:
$$
\alpha_{-2}=\alpha_{-2}^{0}z_{0}+\alpha_{-2}^{1}z_{1}+\alpha_{-2}^{2}z_{2}+\alpha_{-2}^{3}z_{3};
$$
$$
\alpha_{-1}=\alpha_{-1}^{0}z_{0}+\alpha_{-1}^{1}z_{1}+\alpha_{-1}^{2}z_{2}+\alpha_{-1}^{3}z_{3};
$$
$$
\alpha_{0}=\alpha_{0}^{0}z_{0}+\alpha_{0}^{1}z_{1}+\alpha_{0}^{2}z_{2}+\alpha_{0}^{3}z_{3},
$$

The conditions $\alpha_{-1}\circ\alpha_{-2}=0$ and $\alpha_{0}\circ\alpha_{-1}=0$, which guarantee that (\ref{cpx-p3}) is a complex, are equivalent to
\begin{equation}\label{complex condition}
\alpha_{1-i}^{k}\circ\alpha_{-i}^{k}=0\quad\forall k,i\qquad\textnormal{ and } \qquad\alpha_{1-i}^{k}\circ\alpha_{-i}^{l}+\alpha_{1-i}^{l}\circ\alpha_{-i}^{k}=0 \quad \forall i, k\neq l.
\end{equation}
We also have to impose the condition $\ker\alpha_{-1}=\im\alpha_{-2},$ since $\mathcal{H}^{-1}(P^{\bullet})=0.$  

Restricting $P^{\bullet}$ to the plane $\wp\simeq\p2$ we get the following $1$-extended monad on $\wp$:
\begin{equation}\label{Restricted-PEM}
{\small \xymatrix@C-1pc{
V_{1}\otimes\mathcal{O}|_{\wp}(-2)\ar[r]^{\alpha_{-2}|_{\wp}} & V_{1}^{\oplus3}\otimes\mathcal{O}|_{\wp}(-1)\ar[r]^{\alpha_{-1}|_{\wp}\hspace{0.2cm}}& (V_{1}^{\oplus3}\oplus W)\otimes\mathcal{O}|_{\wp}\ar[r]^{\alpha_{0}|_{\wp}} & V_{1}\otimes\mathcal{O}|_{\wp}(1)
}}
\end{equation}
and the maps of this complex are just given by
$$ \alpha_{-2}|_{\wp}=\alpha_{-2}^{0}z_{0}+\alpha_{-2}^{1}z_{1}+\alpha_{-2}^{2}z_{2}; $$
$$ \alpha_{-1}|_{\wp}=\alpha_{-1}^{0}z_{0}+\alpha_{-1}^{1}z_{1}+\alpha_{-1}^{2}z_{2}; $$
$$ \alpha_{0}|_{\wp}=\alpha_{0}^{0}z_{0}+\alpha_{0}^{1}z_{1}+\alpha_{0}^{2}z_{2}. $$

%%%%%%----------------------

The resolution and the monad associated to the perfect extended monad $P^{\bullet}$ are given by, respectively,
\begin{equation}\label{ass-reso}
{\small \xymatrix@C-0.8pc{
0\ar[r]&V_{1}\otimes\mathcal{O}|_{\wp}(-2)\ar[r]^{\alpha_{-2}|_{\wp}} & V_{1}^{\oplus3}\otimes\mathcal{O}|_{\wp}(-1)\ar[r]^{\hspace{1cm}J_{-1}|_{\wp}}&\mathcal{G}|_{\wp}\ar[r]&0
}}
\end{equation}
\begin{equation}\label{ass-monad}
{\small \xymatrix@C-0.8pc{
\mathcal{G}|_{\wp}\ar[r]^{I_{-1}|_{\wp}\hspace{1cm}}& (V_{1}^{\oplus3}\oplus W)\otimes\mathcal{O}|_{\wp}\ar[r]^{\hspace{0.5cm}\alpha_{0}|_{\wp}} & V_{1}\otimes\mathcal{O}|_{\wp}(1)
} } ~.
\end{equation}

%%%%%%%%%%%%%%%%% ?????????????????????????????????????????????????????????????

\begin{lemma}
The sheaf $\mathcal{G}|_{\wp}$ is locally free and satisfies
\begin{itemize}
\item[(i)] $\ho^{0}(\wp,\mathcal{G}|_{\wp})=\ho^{1}(\wp,\mathcal{G}|_{\wp})=\ho^{2}(\wp,\mathcal{G}|_{\wp})=0;$
\item[(ii)] $\ho^{1}(\wp,\mathcal{G}|_{\wp}^{\ast})=\ho^{2}(\wp,\mathcal{G}|_{\wp}^{\ast})=0$, and $h^{0}(\wp,\mathcal{G}|_{\wp}^{\ast})=3c.$
\end{itemize}
\end{lemma}

\begin{proof}
Taking the restriction of the display of the perfect monad to the plane $\wp$ one has $\mathcal{I}|_{\wp}=\mathcal{O}|_{\wp},$ since $supp(Z)\cap\wp=\emptyset.$ Moreover, from the lowest row of the restricted display, namely
$$ 0\to\mathcal{O}|_{\wp}\to Q|{\wp}\to V_{1}\otimes\mathcal{O}|_{\wp}\to0, $$
it follows that $Q|_{\wp}$ is a locally free sheaf. Furthermore, from the middle column of the restricted display, namely
$$ 0\to\mathcal{G}|_{\wp}\to(V_{1}^{\oplus3}\oplus W)\otimes\mathcal{O}|_{\wp}\to Q|_{\wp}\to0, $$
it also follows the sheaf $\mathcal{G}|_{\wp}$ is locally free.

The first item follows from the long exact sequence in cohomology of the associated resolution \eqref{ass-reso} and the fact that $\ho^{i}(\wp,\mathcal{O}|_{\wp}(k))=0,$ for $i=0,1,2$ and $k=-1,-2.$

For the second item, dualize the exact sequence \eqref{ass-reso} and apply the global sections functor $\Gamma$ to obtain the exact sequence
{\small
\begin{equation}\label{dim-count}
0\to\ho^{0}(\wp,\mathcal{G}|_{\wp}^{\ast})\to (V_{1}^{\ast})^{\oplus3}\otimes\ho^{0}(\wp,\mathcal{O}|_{\wp}(1))\to V_{1}^{\ast}\otimes\ho^{0}(\wp,\mathcal{O}|_{\wp}(2))\to\ho^{1}(\wp,\mathcal{G}|_{\wp}^{\ast})\to0.
\end{equation}}
and $\ho^{2}(\wp,\mathcal{G}|_{\wp}^{\ast})=0,$ since 
$\ho^{1,2}(\wp,\mathcal{O}|_{\wp}(1))=\ho^{1,2}(\wp,\mathcal{O}|_{\wp}(1))=0$.
On the other hand, from the dual display of the associated monad \eqref{ass-monad} one has the exact sequence
\begin{equation}\label{dual-mid-col}
0\to Q|_{\wp}^{\ast}\to(V_{1}^{\ast}\oplus W^{\ast})\otimes\mathcal{O}|_{\wp}\to\mathcal{G}|_{\wp}^{\ast}\to0
\end{equation}
where $Q:=\coker\alpha_{-1}.$ Moreover $ Q|_{\wp}^{\ast}=V_{1}^{\ast}\otimes\mathcal{O}|_{\wp}(1)\oplus\mathcal{O}|_{\wp}$ since $ Q|_{\wp}\in\ext^{1}(V_{1}\otimes\mathcal{O}|_{\wp}(1),\mathcal{O}|_{\wp})=V_{1}^{\ast}\otimes\ho^{1}(\wp,\mathcal{O}|_{\wp}(-1))=0.$ Then, from the long exact sequence in cohomology associated to \eqref{dual-mid-col}, it follows that $\ho^{0}(\wp,\mathcal{G}|_{\wp}^{\ast})$ fits in the exact sequence
\begin{equation}\label{fitting1}
0\to\mathbb{C}\to (V_{1}^{\ast})^{\oplus3}\oplus W\to\ho^{0}(\wp,\mathcal{G}|_{\wp}^{\ast})\to0,
\end{equation}
hence $h^{0}(\wp,\mathcal{G}|_{\wp}^{\ast})=3c$ and from \eqref{dim-count} it follows that $h^{1}(\wp,\mathcal{G}|_{\wp}^{\ast})=0.$
\end{proof}

Remark that the sequence \eqref{dim-count} becomes just
\begin{equation}\label{dual}
0\to\ho^{0}(\wp,\mathcal{G}|_{\wp}^{\ast})\stackrel{i}{\to}(V_{1}^{\ast})^{\oplus3}\oplus(V_{1}^{\ast})^{\oplus3}\oplus(V_{1}^{\ast})^{\oplus3}\stackrel{j}{\to} (V_{1}^{\ast})^{\oplus3}\oplus(V_{1}^{\ast})^{\oplus3}\to0,
\end{equation}
since $\ho^{0}(\wp,\mathcal{O}|_{\wp}(1))\simeq\mathbb{C}^{3},$ and
$\ho^{0}(\wp,\mathcal{O}|_{\wp}(2))\simeq\mathbb{C}^{6}.$ So one can
identify $\ho^{0}(\wp,\mathcal{G}|_{\wp}^{\ast})$ with
$(V_{1}^{\ast})^{\oplus3}.$ Furthermore, by \eqref{fitting1}, one
can identify $W=\ho^{0}(\wp,\mathcal{I}_{Z}|_{\wp})\cong\mathbb{C}$, since $Z\cap\wp=\emptyset$.

Combining sequences \eqref{dual} and \eqref{fitting1}, and dualizing the resulting sequence one gets
\begin{equation}\label{fitting2}
0\to V_{1}^{\oplus3}\oplus V_{1}^{\oplus3}\stackrel{i}{\to}V_{1}^{\oplus3}\oplus V_{1}^{\oplus3}\oplus V_{1}^{\oplus3}\stackrel{j}{\to}V_{1}^{\oplus3}\oplus W\to\mathbb{C}\to0,
\end{equation}
The maps $i$ and $j$ are just $\ho^{0}(\alpha_{-2})$ and $\ho^{0}(\alpha_{-1}),$ respectively. Thus we have
$$
\ker\ho^{0}(\alpha_{-2})=\ker\alpha_{-2}^{0}\cap\ker\alpha_{-2}^{1}\cap\ker\alpha_{-2}^{2}=\{0\},
$$
and
$$
\textnormal{ker}\ho^{0}(^{t}\alpha_{-1})=\ker^{\quad t}\alpha_{-1}^{0}\cap\ker^{\quad t}\alpha_{-1}^{1}\cap\ker^{\quad t}\alpha_{-1}^{2}=\mathbb{C}.
$$
The subscript $t,$ in the last equation, stands for transposition.
Remark also that the sequence \eqref{fitting2} reflects the fact the complex \eqref{Restricted-PEM} is exact at degree $-1,$ i.e., $\alpha_{-1}\circ\alpha_{-2}=0.$

We can then choose the maps $\alpha_{-1}^{j}$ in the following way. First,
$$\alpha_{-2}^{0},\hspace{0.3cm}\alpha_{-2}^{1},\hspace{0.3cm}\alpha_{-2}^{2}:V_{1}\to V_{1}\oplus V_{1}\oplus V_{1}, $$
with:
\begin{equation}
\alpha_{-2}^{0}=\left(\begin{array}{l} \hspace{0.1cm}0\\\hspace{0.1cm}0\\ \mathbb{I}_{V_{1}}\end{array}\right) \quad 
\alpha_{-2}^{1}=\left(\begin{array}{l}\hspace{0.3cm}0 \\ -\mathbb{I}_{V_{1}} \\ \hspace{0.3cm}0\end{array}\right) \quad
\alpha_{-2}^{2}=\left(\begin{array}{l} \mathbb{I}_{V_{1}} \\ \hspace{0.1cm}0 \\ \hspace{0.1cm}0
\end{array}\right) ,
\end{equation}
and where $\mathbb{I}_{V_{1}}$ denotes the identity in $\End(V_{1})$.

One also has
$$\alpha_{-1}^{0},\hspace{0.3cm}\alpha_{-1}^{1},\hspace{0.3cm}\alpha_{-1}^{2}:V_{1}\oplus V_{1}\oplus V_{1}\to V_{1}\oplus V_{1}\oplus V_{1}\oplus\mathbb{C} $$
given by
\begin{equation} \begin{split}
& \alpha_{-1}^{0} = \left(\begin{array}{lll}
\hspace{0.1cm}0 & \hspace{0.1cm}0&0\\ \mathbb{I}_{V_{1}}&\hspace{0.1cm}0&0\\ \hspace{0.1cm}0&\mathbb{I}_{V_{1}}&0 \\ \hspace{0.1cm}0&\hspace{0.1cm}0&0
\end{array}\right)  \quad\quad
\alpha_{-1}^{1} = \left(\begin{array}{lll}
-\mathbb{I}_{V_{1}} & 0&\hspace{0.1cm}0\\ \hspace{0.3cm}0&0&\hspace{0.1cm}0\\ \hspace{0.3cm}0&0& \mathbb{I}_{V_{1}}\\ \hspace{0.3cm}0&0&\hspace{0.1cm}0
\end{array}\right) \\
&\hspace{2cm} \alpha_{-1}^{2} = \left(\begin{array}{lll}
0 & -\mathbb{I}_{V_{1}}&\hspace{0.3cm}0\\ 0&\hspace{0.3cm}0&-\mathbb{I}_{V_{1}}\\ 0&\hspace{0.3cm}0&\hspace{0.3cm}0 \\ 0&\hspace{0.3cm}0&\hspace{0.3cm}0
\end{array}\right).
\end{split} \end{equation}

Finally, for
$$ \alpha_{0}^{0},\hspace{0.3cm}\alpha_{0}^{1},\hspace{0.3cm}\alpha_{0}^{2}:
V_{1}\oplus V_{1}\oplus V_{1}\oplus\mathbb{C}\to V_{1} $$
one has 
\begin{equation} \begin{split}
& \alpha_{0}^{0} = \left(\begin{array}{llll} -\mathbb{I}_{V_{1}}&0 & 0&0\end{array}\right)
\quad\alpha_{0}^{1}=\left(\begin{array}{llll} 0&-\mathbb{I}_{V_{1}} & 0&0\end{array}\right) \\
&\hspace{2cm}\alpha_{0}^{2}=\left(\begin{array}{llll} 0&0 & -\mathbb{I}_{V_{1}}&0\end{array}\right).
\end{split} \end{equation}

\bigskip

Now, to complete our construction, we have to add the maps $\alpha_{-2}^{3},$ $\alpha_{-1}^{3}$ and $\alpha_{0}^{3}.$ such that conditions \eqref{complex condition} are satisfied. By putting
\begin{equation}\label{alpha3a}
\alpha_{-2}^{3}=\left(\begin{array}{l} -B_{2} \\ \hspace{0.3cm}B_{1}\\ -B_{0}\end{array}\right);
\quad\alpha_{-1}^{3}=\left(\begin{array}{lll} \hspace{0.3cm}B_{1} &\hspace{0.3cm}B_{2} &\hspace{0.3cm}0\\ -B_{0}&\hspace{0.4cm}0&\hspace{0.3cm} B_{2}\\ \hspace{0.4cm}0&-B_{0}&-B_{1} \\ \hspace{0.4cm}0&\hspace{0.4cm}0&\hspace{0.4cm}0\end{array}\right);
\alpha_{0}^{3}=\left(\begin{array}{llll} B_{0}&B_{1} & B_{2}&I\end{array}\right),
\end{equation}
where $B_{i}\in\End(V_{1})$ and $I\in\Hom(\mathbb{C},V_{1}),$ then all the equations are satisfied, since $\alpha_{-1}^{3}\circ\alpha_{-2}^{3}=0$ and $\alpha_{0}^{3}\circ\alpha_{-1}^{3}=0$ are equivalent to
\begin{equation}\label{commutation}
[B_{0},B_{1}]=0;\quad[B_{0},B_{2}]=0;\quad[B_{1},B_{2}]=0,
\end{equation}

Summing up what we have done so far, for a given $0$-dimensional subscheme $Z\subset\p3\setminus\wp$ we have constructed a perfect extended monad $P^\bullet$ of the form

\begin{equation}\label{perfect-monad -ideal}
{\small \xymatrix@C-1pc{
V_{1}\otimes\mathcal{O}_{\p3}(-2)\ar[r]^{\alpha_{-2}} & V_{1}^{\oplus3}\otimes\mathcal{O}_{\p3}(-1)\ar[r]^{\alpha_{-1}\hspace{0.5cm}}& (V_{1}^{\oplus3}\oplus W)\otimes\mathcal{O}_{\p3}\ar[r]^{\alpha_{0}\hspace{0.2cm}} & V_{1}\otimes\mathcal{O}_{\p3}(1)
}}
\end{equation}

\noindent where the maps  maps $\alpha_{-2},$ $\alpha_{-1}$ and $\alpha_{0}$ are given by
\footnote{ We omit writing the identity in front of the coordinates so $z_{i}\mathbb{I}_{V_{1}}$ will just be written $z_{i}$}:
\begin{equation}\label{alpha3}
\begin{split}
\alpha_{-2}=\left(\begin{array}{l} -B_{2}z_{3}+z_{2} \\\hspace{0.3cm} B_{1}z_{3}-z_{1}\\ -B_{0}z_{3}+z_{0}\end{array}\right);&
\quad\alpha_{-1}=\left(\begin{array}{lll} \hspace{0.3cm}B_{1}z_{3}-z_{1} & \hspace{0.3cm}B_{2}z_{3}-z_{2} &\hspace{0.8cm}0\\ -B_{0}z_{3}+z_{0}&\hspace{0.8cm}0& \hspace{0.3cm}B_{2}z_{3}-z_{2}\\ \hspace{0.8cm}0&-B_{0}z_{3}+z_{0}&-B_{1}z_{3}+z_{1} \\ \hspace{0.8cm}0&\hspace{0.8cm}0&\hspace{0.8cm}0\end{array}\right); \\ 
&\alpha_{0}=\left(\begin{array}{llll} B_{0}z_{3}-z_{0}&B_{1}z_{3}-z_{1} & B_{2}z_{3}-z_{2}&Iz_{3}\end{array}\right).
\end{split}
\end{equation}
such that
\begin{equation}%\label{commutation}
[B_{0},B_{1}]=0;\quad[B_{0},B_{2}]=0;\quad[B_{1},B_{2}]=0,
\end{equation}
Moreover the ADHM datum
$ (B_{0},B_{1},B_{2},I)\in\End(V_{1})^{\oplus3}\oplus\Hom(\mathbb{C},V_{1}) $ is indeed stable. Such claim will follow from the following observation.

\begin{lemma}
The map $\alpha_{0}$ given above is surjective if and only if the ADHM datum $(B_{0},B_{1},B_{2},I)$ is stable.
\end{lemma}
\begin{proof}
The proof is a straightforward generalization of \cite[Lemma 2.7 (2)]{N2}.
\end{proof}

%%%%%%%%%%%%%%%%% REFORMULATE %%%%%%%%%%%%%%%%%%%%%%%%%%%%%%%%%%%%%%%%%%%%%%%%%

\begin{theorem}[Inverse construction]\label{reconstruction}
To a stable ADHM datum \linebreak $X=(B_{0},B_{1},B_{2},I)\in\mathcal{V}(3,c)^{st}$ one can associate the perfect extended monad \eqref{perfect-monad -ideal} with maps $\alpha_{-2},\alpha_{-1},\alpha_{0}$ given as in \eqref{alpha3}, such that its cohomology is an ideal sheaf whose restriction to $\mathbb{C}^{3}=\mathbb{P}^{3}\backslash\wp$ is isomorphic to the one given by Theorem\ref{Corresp}.
\end{theorem}

\begin{proof}
The proof is an easy generalization of \cite[Proposition 2.8]{N2}
\end{proof}

The automorphisms of $P^{\bullet}$ are clearly given by the action of the group $GL(V_{1}).$ Since, by Corollary \ref{fully-faithfull}, the cohomology functor is fully faithful, we recover the correspondence, given in Section \ref{Matrix-Para}, between equivalence classes of ideal sheaves
$\mathcal{I}_{Z}$ and the space $\mathcal{M}(3,c)$ defined as the quotient
$\mathcal{V}(3,c)^{st}/GL(V_{1})$, in the $3$-dimensional case.

\vspace{0.5cm}
We complete this section by writing down the maps $\alpha_{0}$ and $\alpha_{-1}$ in the more general $n$-dimensional case. Starting with a hyperplane $\wp\subset\pn$ and a $0$-dimensional subscheme $Z\subset\pn\setminus\wp$, the maps $\alpha_{-i}$ in the corresponding perfect extended monad can also be constructed as done above for the $3$-dimensional case:
\begin{equation}\begin{split}
&\alpha_{0}=\left(\begin{array}{lllll} B_{0}z_{n}-z_{0}&B_{1}z_{n}-z_{1} & \cdots&B_{n-1}z_{n}-z_{n-1}&Iz_{n}\end{array}\right). \\
&\alpha_{-1}=\left(\begin{array}{llll} A_{0} & A_{1}  & \cdots & A_{n-2} \\ \hspace{0.1cm}0 & \hspace{0.1cm}0 & \cdots & \hspace{0.3cm}0 \end{array}\right).
\end{split}
\end{equation}
where each block $A_{i},$ $0\leq i\leq n-2$ is an $[ n\cdot c \times (n-i-1)\cdot c]$-matrix of the form
{\small
\begin{displaymath}
A_{i}=\left( \begin{array}{lllll}
\hspace{1cm}0&\hspace{1cm}0&\hspace{1cm}0&\cdots&\hspace{1cm}0\\
\hspace{1cm}0 &\hspace{1cm}0&\hspace{1cm}0&\cdots& \hspace{1cm}0\\
\hspace{1cm}\vdots &\hspace{1cm}\vdots&\hspace{1cm}\vdots&\cdots&\hspace{1cm}\vdots \\
\hspace{1cm}0&\hspace{1cm}0&\hspace{1cm}0&\cdots&\hspace{1cm}0\\
B_{i+1}z_{n}-z_{i+1}&B_{i+2}z_{n}-z_{i+2}&B_{i+3}z_{n}-z_{i+3}&\cdots&B_{n-1}z_{n}-z_{n-1} \\
-B_{i}z_{n}+z_{i} &\hspace{1cm}0&\hspace{1cm}0&\cdots&\hspace{1cm}0 \\
 \hspace{1cm}0&-B_{i}z_{n}+z_{i}&\hspace{1cm}0&\cdots&\hspace{1cm}0 \\
 \hspace{1cm}0&\hspace{1cm}0&-B_{i}z_{n}+z_{i}&\cdots&\hspace{1cm}0 \\
 \hspace{1cm}0&\hspace{1cm}0&\hspace{1cm}0&\cdots&\hspace{1cm}0 \\
\hspace{1cm}\vdots&\hspace{1cm}\vdots&\hspace{1cm}\vdots&\ddots&\hspace{1cm}0 \\
 \hspace{1cm}0&\hspace{1cm}0&\hspace{1cm}0&\cdots&-B_{i}z_{n}+z_{i}
\end{array}\right)
\end{displaymath}
}
The first non vanishing line in $A_{i}$ is $(i+1)$-th one. For instance, when $n=3$ there are two blocks $$A_{0}=\left(\begin{array}{ll} \hspace{0.3cm}B_{1}z_{3}-z_{1} & \hspace{0.3cm}B_{2}z_{3}-z_{2} \\ -B_{0}z_{3}+z_{0}&\hspace{0.8cm}0\\ \ \hspace{0.8cm}0&-B_{0}z_{3}+z_{0} \end{array}\right)
\hspace{1cm} A_{1}=\left(\begin{array}{l}\hspace{0.8cm}0\\ \hspace{0.3cm}B_{2}z_{3}-z_{2}\\-B_{1}z_{3}+z_{1} \end{array}\right),$$ of respective sizes $[3\cdot c \times 2\cdot c]$ and $[3\cdot c \times 1\cdot c].$
One can similarly show that $\alpha_{0}\circ\alpha_{-1}=0\Leftrightarrow [B_{i},B_{j}]=0,$ for all $0\leq i,j\leq n-1.$ and that the map $\alpha_{0}$ is surjective if and only if the ADHM datum
$(B_{0},\dots,B_{n-1}, I)$ is stable.
 
Once again, this reflects the set theoretic bijection between the Hilbert scheme of length $c$ zero-dimensional subschemes of $\mathbb{C}^{n}\simeq\pn\setminus\wp$ and the quotient space
$\mathcal{M}(n,c):=\mathcal{V}(n,c)^{st}/GL(V_{1}).$

%%%%%%%%%%%%%%%%%%%%%%%%%%%%%%%%%%%%%%%%%%%%%%%%%%%%%%%%%%%%%%%%%%

\subsection{Representability of the Hilbert functor of points}\label{Rep-mod-functor}

Let us start this Section by introducing notation; for every two sheaves, $\mathcal{F}$ on $\pn$ and $\mathcal{G}$ on a scheme $S$, we put $\mathcal{F}\boxtimes \mathcal{G}:=p^{\ast}\mathcal{F}\otimes q^{\ast}\mathcal{G},$ where $p:\pn\times S\longrightarrow\pn$ is the projection on the first factor and $q$ is the projection $\mathbb{P}^{n}\times S\longrightarrow S$ on the second one. We also denote by $k(s)$ the residue field of a closed point $s\in S$.

Using the ingredients developed in the previous sections, we now proceed to prove that
$\mathcal{M}(n,c)$ represents the Hilbert functor $ \mathcal{H}\textnormal{ilb}_{\C^n}^{[c]}:\mathfrak{S}ch\to\mathfrak{S}et $ from the category of schemes $\mathfrak{S}ch$ to the category of sets $\mathfrak{S}et$, which associates to every scheme $S$ the set
$$ \mathcal{H}\textnormal{ilb}_{\C^n}^{[c]}(S) := 
\left\{
\hspace{-0.2cm}\begin{array}{c} 
\phantom{.} \\ Z\subset\C^n\times S \\ \phantom{.} \\ \phantom{.} \end{array} \right.
\left|
\begin{array}{l} 
Z\textnormal{ is a closed subscheme,} \\ 
\begin{array}{ccc}
Z & \hookrightarrow & \C^n\times S \\
\pi\downarrow & &\downarrow q \\ S & \simeq & S
\end{array} \hspace{0.4cm} \textnormal{with }\pi\textnormal{ flat, and} \\
\chi(\mathcal{O}_{\pi^{-1}(s)}\otimes\mathcal{O}_{\C^n}(m))=c \quad, \quad \forall m\in\Z.
\end{array}
\right\}
$$
of flat families of $0$-dimensional subschemes of $\C^n$.

For any noetherian scheme $S$ of finite type over the field of
complex numbers $\C$, consider the following diagram:
\begin{displaymath}
\xymatrix@1{\mathbb{P}^{n}\times\mathbb{P}^{n}\times S \ar[r]^{\quad pr_{13}}\ar[d]_{pr_{23}} & \mathbb{P}^{n}\times S \ar[d]^{q} \\
\mathbb{P}^{n}\times S \ar[r]_{q}& S
}
\end{displaymath}
and the relative Euler sequence:
$$
0\longrightarrow\mathcal{O}_{\mathbb{P}^{n}\times S}(-1)\longrightarrow\mathcal{O}^{\oplus(n+1)}_{\mathbb{P}^{n}\times
S}\longrightarrow T\mathbb{P}^{n}(-1)\boxtimes\mathcal{O}_{S}\longrightarrow 0
$$
where $T\mathbb{P}^{n}(-1)$ is tangent bundle. One has the following:

\begin{theorem}[Relative Beilinson's Theorem.]
For every coherent sheaf $\mathcal{E}$ on $\mathbb{P}^{n}\times S$ there is a spectral sequence $E^{i,j}_{r}$ with $E_{1}$-term $$E_{1}^{i,j}=\mathcal{O}_{\mathbb{P}^{n}}(i)\boxtimes\mathcal{R}^{j}q_{\ast}(\mathcal{E}\otimes\Omega_{\mathbb{P}^{n}\times S/S}^{-i}(-i))$$ which converges to $$E_{\infty}^{i,j}=\left\{\begin{array}{ll} \mathcal{E}& i+j=0 \\ 0 & {\rm otherwise.} \end{array}\right.$$
\end{theorem}

Let $\mathcal{J}$ be an $S$-flat family of ideal sheaves of $0$-dimensional subschemes of $\pn$ of length $c$ , for a noetherian scheme $S$ of finite type.

\begin{theorem}\label{hilbert functor}
There exists an $1$-extended monad given by

{\small $$ 0\to\mathcal{O}_{\pn}(1-n)\boxtimes
\mathcal{R}^{1}q_{\ast}(\mathcal{J}\otimes\Omega_{\pn\times S /S}^{n}(n-1)) \! \to \!
\mathcal{O}_{\pn}(2-n)\boxtimes
\mathcal{R}^{1}q_{\ast}(\mathcal{J}\otimes\Omega_{\pn\times S /S}^{n-1}(n-2)) \! \to \! \cdots
 $$
\begin{equation}\label{Universal-extended}
\cdots\to \mathcal{O}_{\pn}\boxtimes
\mathcal{R}^{1}q_{\ast}(\mathcal{J}\otimes\Omega_{\pn\times S /S}^{1}) \to 
\mathcal{O}_{\pn}(1)\boxtimes
\mathcal{R}^{1}q_{\ast1}(\mathcal{J}\otimes p^{\ast}\mathcal{O}_{\pn}(-1))\to0
\end{equation} }
such that its cohomology is exactly the family $\mathcal{J}.$
\end{theorem}

\begin{proof}
By the relative Beilinson theorem, we only need the $S$-flatness of $\mathcal{J}$ and the fact that at point $s\in S$ one has
$$\mathcal{R}^{1}q_{\ast}(\mathcal{J}\otimes\Omega_{\pn\times S/S}^{-i}(1-i))\otimes k(s)\simeq
\ho^{1}(\pn,\mathcal{I}_{Z(s)}\otimes\Omega_{\pn}^{-i}(1-i)),$$
where $Z(s)$ is the $0$-dimensional subscheme of $\pn$ corresponding to the point $s\in S$. 
The rest of the proof follows from the vanishing properties of Lemma \ref{vanishing}.
\end{proof}

Therefore, on every point $s\in S,$ one has a perfect extended monad $ P^{\bullet}(s)$ given by
$$ \ho^{1}(\mathcal{I}_{Z(s)}\otimes\Omega_{\pn}^{n}(n-1))\otimes\mathcal{O}_{\pn}(1-n) \to
\ho^{1}(\mathcal{I}_{Z(s)}\otimes\Omega_{\pn}^{n-1}(n-2))\otimes\mathcal{O}_{\pn}(2-n) \to \cdots $$
$$ \cdots \to \ho^{1}(\mathcal{I}_{Z(s)}\otimes\Omega_{\pn}^{1})\otimes\mathcal{O}_{\pn}(-1) \to 
\ho^{1}(\mathcal{I}_{Z(s)}\otimes\mathcal{O}_{\pn}(-1))\otimes\mathcal{O}_{\pn}(1) $$

Moreover, in the case of the space $\mathcal{V}(n,c)^{st},$ defined in Section \ref{Matrix-Para}, just after Definition \ref{stability}, one has the \emph{universal extended monad}
{\small $$ 0\to\mathcal{O}_{\pn}(1-n)\boxtimes (V_{1}\otimes\mathcal{O}_{\mathcal{V}(n,c)^{st}}) \to 
\mathcal{O}_{\pn}(2-n)\boxtimes(V_{1}^{\oplus\binom{n}{n-1}}\otimes\mathcal{O}_{\mathcal{V}(n,c)^{st}})
\to \cdots $$
\begin{equation}\label{Universal-extended2}
\cdots \to 
\mathcal{O}_{\pn}\boxtimes((V_{1}^{\oplus n}\oplus W)\otimes\mathcal{O}_{\mathcal{V}(n,c)^{st}})
\to \mathcal{O}_{\pn}(1)\boxtimes(V_{1}\otimes\mathcal{O}_{\mathcal{V}(n,c)^{st}})\to0.
\end{equation}
}
Finally we have the following
\begin{theorem}\label{represent}
The scheme $\mathcal{M}(n,c)$ is a fine moduli space for the Hilbert functor
$\mathcal{H}\textnormal{ilb}_{\C^n}^{[c]}$ of $c$ points  on $\mathbb{C}^{n}$.
\end{theorem}

\begin{proof}
The proof is similar, \emph{mutatis mutandis}, to that of \cite[Theorem 4.2]{HJM}.
\end{proof}

It follows by universality of the Hilbert scheme that

\begin{corollary}\label{identified}
$\Hilb^{[c]}({\mathbb{C}^{n}})\simeq\mathcal{M}(n,c)$ as schemes.
\end{corollary}

%%%%%%%%%%%%%%%%%%%%%%%%%%%%%%%%%%%%%%%%%%%%%%%%%%%%%%%%%%%%%%%%%%%%%%%%%%%%%%%%%%%%%%%%%%%%%%
%%%%%%%%%%%%%%%%%%%%%%%%%%%%%%%%%%%%%%%%%%%%%%%%%%%%%%%%%%%%%%%%%%%%%%%%%%%%%%%%%%%%%%%%%%%%%%

\section{The Hilbert scheme of points on affine varieties}\label{hilb_y}

In this section we realize a scheme-theoretic bijection between ideals of zero dimensional subschemes, with constant Hilbert polynomial $c,$ on affine varieties $\mathbb{Y}$ and points of a subscheme of $\Hilb^{[c]}(\mathbb{C}^{n})$ which is defined out of the ideal associated to $\mathbb{Y}.$ 

Let us consider the following data: Let $\mathbb{Y}=\mathcal{Z}(Z_{\mathbb{Y}})\subset\mathbb{C}^{n}$ be an affine variety, given by the zero locus of the ideal $Z_{\mathbb{Y}}\subsetneq\mathbb{C}[x_{1},\cdots,x_{n}].$ We denote by $A(\mathbb{Y})$ the affine coordinate ring of the variety $\mathbb{Y},$ i.e., $A(\mathbb{Y})=\frac{\mathbb{C}[x_{1},\cdots,x_{n}]}{Z_{\mathbb{Y}}}.$ 

To each stable datum $X=(B_{0},\cdots,B_{n-1},I),$ as defined in Section \ref{Matrix-Para}, one can associate a unique ideal $J\subset\mathbb{C}[x_{1},\cdots,x_{n}]$, up to a $GL(V)$ action, such that the quotient $\frac{\mathbb{C}[x_{1},\cdots,x_{n}]}{J}=V$ is a vector space of dimension $c.$ In other words, $J$ is an ideal corresponding to a zero dimensional subschemes, of length $c,$ of the affine space $\mathbb{C}^{n}.$ If we define
$$ \mathcal{V}_{\mathbb{Y}}(c)^{st} := \{X=(B_{0},\cdots,B_{n-1},I)\in\mathcal{V}(n,c)^{st}/f(B_{0},\cdots,B_{n-1})=0,\quad\forall f\in Z_{\mathbb{Y}}\}. $$ 
Hence we have the following:
\begin{theorem}\label{correspond2}
There exists a set-theoretical bijection between the quotient space $\mathcal{M}_{\mathbb{Y}}(c) := \mathcal{V}_{\mathbb{Y}}(c)^{st} / GL(V)$ and the Hilbert scheme $\Hilb^{[c]}(\mathbb{Y})$of $c$ points in $\mathbb{Y}.$
\end{theorem}

\subsection{Scheme structure on $\mathcal{M}_{\mathbb{Y}}(c)$}

The scheme structure of $\mathcal{M}_{\mathbb{Y}}(c)$ is given as the following:  For a given datum $X=(B_{0},\cdots,B_{n-1}, I)\in\mathcal{V}(n,c)^{st},$ the map $\phi_{X}'$ is given by $$\phi_{X}':\begin{array}{ccc} A(\mathbb{Y})& \to & V \\ \lbrack p\mod Z_{\mathbb{Y}} \rbrack & \mapsto &\phi'([p\mod(Z_{\mathbb{Y}})]):=[p(B_{0},\cdots,B_{n-1})\mod(Z_{\mathbb{Y}})]I(1).\end{array}$$
By stability, one has $\ker\phi_{X}'=\{f\in A(\mathbb{Y})=\frac{\mathbb{C}[x_{1},\cdots,x_{n}]}{Z_{\mathbb{Y}}} | f(B_{0},\cdots,B_{n-1})=0\}.$ In particular $f(B_{0},\cdots,B_{n-1})=0$ for all $f\in Z_{\mathbb{Y}}.$ Conversely, one can define an ideal $\tilde{Z}_{\mathbb{Y}}$ in the ring of regular functions $\Gamma(\mathcal{V}(n,c)^{st}),$  on $\mathcal{V}(n,c)^{st},$  which is defined by $\tilde{Z}_{\mathbb{Y}}=\{f(B_{0},\cdots,B_{n-1})=0\in\End(V)\textnormal{ for }f\in Z_{\mathbb{Y}}\}.$ Then $\mathcal{V}_{\mathbb{Y}}(c)^{st}$ is the subscheme of $\mathcal{V}(n,c)^{st}$ given by $\tilde{Z}_{\mathbb{Y}}\subset\Gamma(\mathcal{V}(n,c)^{st}).$ Of course, by choosing a basis for $V,$ one has $\End(V)\cong \Mat_{c\times c}(\mathbb{C})$ and every polynomial equation $f\in Z_{\mathbb{Y}}$ gives rise to $c$ polynomial equations in $\tilde{Z}_{\mathbb{Y}}$ given the rows of $f(B_{0},\cdots,B_{n-1}).$ Thus, if $Z_{\mathbb{Y}}$ is generated by $k$ element in $\mathbb{C}[x_{1},\cdots,x_{n}],$ then $\tilde{Z}_{\mathbb{Y}}$ will be generated by, at most, $c\times k$ polynomials in the entries $b_{ab}^{i}$ of $B_{i},$ for $i=0,\cdots,n-1,$ and $a,b=1,\cdots,c.$  

Now, consider  $\mathcal{O}_{\Hilb^{[c]}(\mathbb{C}^{n})}:=\bigoplus_{i\geq0}\Gamma(\mathcal{V}(n,c))^{G,\chi^{i}},$ where $\Gamma(\mathcal{V}(n,c))^{G,\chi^{i}}$ is the ring of equivariant regular functions on $\mathcal{V}(n,c),$ of weight $i$ with respect to the character $\chi$ corresponding to the $GL(V)$-action, as defined in Section \ref{Matrix-Para}. One can form the space {\small $$\mathcal{M}_{\mathbb{Y}}(c) :=
{\rm Proj}\left(\bigoplus_{i\geq0}\left(\frac{\Gamma(\mathcal{V}(n,c))}{\tilde{Z}_{\mathbb{Y}}}\right)^{G,\chi^{i}} \right) \hookrightarrow\Hilb^{[c]}(\mathbb{C}^{n}):=
{\rm Proj}\left(\bigoplus_{i\geq0}\Gamma(\mathcal{V}(n,c))^{G,\chi^{i}} \right).$$ }

Hence, $\mathcal{M}_{\mathbb{Y}}(c)$ is a closed subscheme of the Hilbert scheme of points $\Hilb^{[c]}(\mathbb{C}^{n}),$ and represents, of course a subfunctor, $\mathfrak{M}_{\mathbb{Y}}^{[c]}(\cdot)$ of the Hilbert functor $\mathcal{H}ilb_{\mathbb{C}^{n}}^{[c]}(\cdot)$

\subsection{The schematic isomorphism $\mathcal{M}_{\mathbb{Y}}(c)\cong\Hilb^{[c]}(\mathbb{Y})$}

Recall that the Hilbert scheme of points $\Hilb^{[c]}(\mathbb{Y})$ represents the functor $\mathcal{H}ilb_{\mathbb{Y}}^{[c]}(\cdot):\mathfrak{S}ch\to\mathfrak{S}et$ which associates to any noetherian scheme of finite type $S$ the set
$$ \mathcal{H}ilb_{\mathbb{Y}}^{[c]}(S) := 
\left\{
\hspace{-0.2cm}\begin{array}{c} 
\phantom{.} \\ Z\subset\mathbb{Y}\times S \\ \qquad\cap \\ \qquad\mathbb{C}^{n}\times S\end{array} \right.
\left|
\begin{array}{l} 
Z\textnormal{ is a closed subscheme,} \\ 
\begin{array}{ccc}
Z & \hookrightarrow & \mathbb{Y}\times S \\
\pi\downarrow & &\downarrow q \\ S & \simeq & S
\end{array} \hspace{0.4cm} \textnormal{with }\pi\textnormal{ flat, and} \\
\chi(\mathcal{O}_{\pi^{-1}(s)}\otimes\mathcal{O}_{\mathbb{Y}}(m))=c \quad, \quad \forall m\in\Z.
\end{array}
\right\}
$$
of flat families of $0$-dimensional subschemes of $\mathbb{Y}\stackrel{i}{\hookrightarrow}\mathbb{C}^{n}$ of length $c.$
In particular there exists a universal flat family $\mathfrak{X}\hookrightarrow\mathbb{Y}\times\Hilb^{[c]}(\mathbb{Y})$ of $0$-dimensional subschemes of length $c$ on $\mathbb{Y}.$ Moreover, $\mathfrak{X}$ is a flat subfamily of the universal family $\mathfrak{F}\subset\mathbb{C}^{n}\times\Hilb^{[c]}(\mathbb{C}^{n})$, of $0$-dimensional subschemes of length $c$ in $\mathbb{C}^{n}.$ This follows from the fact that $\mathcal{H}ilb_{\mathbb{Y}}^{[c]}(\cdot)\stackrel{\tilde{h}^{\natural}}{\to}\mathcal{H}ilb_{\mathbb{C}^{n}}^{[c]}(\cdot)$ is a subfunctor and $\mathfrak{X}$ is just the restriction $\tilde{h}^{\ast}\mathfrak{F}$ of $\mathfrak{F},$ for the morphism $\tilde{h}:=i\times h:\mathbb{Y}\times\Hilb^{[c]}(\mathbb{Y})\hookrightarrow\mathbb{C}^{[c]}\times\Hilb^{[c]}(\mathbb{C}^{n})$.
Furthermore, since the scheme $\mathcal{M}_{\mathbb{Y}}(c)$ also parametrizes $0$-dimensional subschemes of length $c$ in $\mathbb{Y}\hookrightarrow\mathbb{C}^{n},$ then any flat family $\mathfrak{Z}\subset\mathbb{Y}\times\mathcal{M}_{\mathbb{Y}}(c)$ is the pull-back of the family $\mathfrak{X}$ under a morphism $\mathbb{Y}\times\mathcal{M}_{\mathbb{Y}}(c)\stackrel{id_{\mathbb{Y}}\times f}{\to}\mathbb{Y}\times\Hilb^{[c]}(\mathbb{Y})$. One can resume the above situation in the following diagram:
{\small\begin{equation*}
\xymatrix@R-1.5pc@C-1.5pc{ \mathfrak{X}\ar[dd]\ar[drr]& & & &\mathfrak{Z}\ar[dd]\ar[dll]\ar[llll] \\
 & &\mathfrak{F}\ar[dd] & & \\
 \mathbb{Y}\times\Hilb^{[c]}(\mathbb{Y})\ar@{^{(}->}[rrd]_{\tilde{\beta}}\ar[dd]& &\ar@{-->}[ll]_{\tilde{f}} & &\mathbb{Y}\times\mathcal{M}_{\mathbb{Y}}(c)\ar@{--}[ll]\ar@{_{(}->}[lld]^{\tilde{\alpha}}\ar[dd] \\
 & & \mathbb{C}^{n}\times\Hilb^{[c]}(\mathbb{C}^{n})\ar[dd] & & \\
 \Hilb^{[c]}(\mathbb{Y})\ar@{^{(}->}[rrd]_{\beta}& & \ar@{-->}[ll]_{f}& & \mathcal{M}_{\mathbb{Y}}(c)\ar@{--}[ll]\ar@{_{(}->}[lld]^{\alpha} \\
 & & \Hilb^{[c]}(\mathbb{C}^{n})& & 
}
\end{equation*}}
where $\tilde{\alpha}:=i\times\beta,$ $\tilde{\beta}:=i\times\beta$ and 
$\tilde{f}:=id_{\mathbb{Y}}\times f.$  
On the other hand, $\mathcal{M}_{\mathbb{Y}}(c)$ represents a closed subfunctor of the Hilbert functor $\mathcal{H}ilb_{\mathbb{C}^{n}}^{[c]}(\bullet)$ such that, for any closed point $Spec(k),$ the following diagram commutes:
\begin{equation}\label{points}\xymatrix@R-1pc@C-1pc{ &Spec(k)\ar[rd]^{\rho_{\alpha}}\ar[ld]_{\rho_{\alpha}}\ar@{-}[d]^{\rho}& \\
\Hilb^{[c]}(\mathbb{Y})\ar@{^{(}->}[rd]^{\beta}&\ar[d]\ar@{-->}[l]&\mathcal{M}_{\mathbb{Y}}(c)\ar@{--}[l]_f\ar@{_{(}->}[ld]_{\alpha} \\
&\Hilb^{[c]}(\mathbb{C}^{n})& \\
}\end{equation}
One has $\Hom(Spec(k),\mathcal{M}_{\mathbb{Y}}(c))\cong\mathcal{H}ilb_{\mathbb{Y}}^{[c]}(Spec(k)),$ since $f$ is an isomorphism on points, as it follows from Theorem \ref{correspond2}. 

On the other hand, let us consider the restriction of the universal monad to $\mathbb{Y}\times\mathcal{V}_{\mathbb{Y}}(c)^{st}.$ This gives the following extended monad 

{\small $$ 0\to\mathcal{O}_{\mathbb{Y}}(1-n)\boxtimes (V_{1}\otimes\mathcal{O}_{\mathcal{V}_{\mathbb{Y}}(c)^{st}}) \to 
\mathcal{O}_{\mathbb{Y}}(2-n)\boxtimes(V_{1}^{\oplus\binom{n}{n-1}}\otimes\mathcal{O}_{\mathcal{V}_{\mathbb{Y}}(c)^{st}})
\to \cdots $$
\begin{equation}\label{Univ-ext-rest}
\cdots \to 
\mathcal{O}_{\mathbb{Y}}\boxtimes((V_{1}^{\oplus n}\oplus W)\otimes\mathcal{O}_{\mathcal{V}_{\mathbb{Y}}(c)^{st}})
\to \mathcal{O}_{\mathbb{Y}}(1)\boxtimes(V_{1}\otimes\mathcal{O}_{\mathcal{V}_{\mathbb{Y}}(c)^{st}})\to0.
\end{equation} }
 We denote by $\tilde{\mathfrak{Z}}$ the family of ideals which arises as its cohomology. The restriction at any point $Spec(k)\to\mathcal{V}_{\mathbb{Y}}(c)^{st}$ one has an $(n-2)$-extended monad whose cohomology is an ideal of $c$ points on $\mathbb{Y}$ represented by the stable datum $(B_{1},\cdots,B_{n},I)$ which satisfies $f(B_{1},\cdots,B_{n})=0,$ for all $f\in Z_{\mathbb{Y}}.$

Now for any noetherian scheme of finite type $S,$ parametrizing zero-dimensional schemes of $\mathbb{Y},$ suppose there exists a flat family of ideals $\eta$ on $\mathbb{Y}\times S.$  We consider an open cover $\{S_{i}\}_{i\in J}$ of $S.$ Then, for a fixed $i\in J,$ $\eta_{i}:=\eta({\mathbb{Y}\times S_{i}})$ is the cohomology of the following $(n-2)$-extended monad on $S_{i}$ 
{\small $$ 0\to\mathcal{O}_{\mathbb{Y}}(1-n)\boxtimes (V_{1}\otimes\mathcal{O}_{S_{i}}) \to 
\mathcal{O}_{\mathbb{Y}}(2-n)\boxtimes(V_{1}^{\oplus\binom{n}{n-1}}\otimes\mathcal{O}_{S_{i}})
\to \cdots $$
\begin{equation}%\label{Univ-ext-rest}
\cdots \to 
\mathcal{O}_{\mathbb{Y}}\boxtimes((V_{1}^{\oplus n}\oplus W)\otimes\mathcal{O}_{S_{i}})
\to \mathcal{O}_{\mathbb{Y}}(1)\boxtimes(V_{1}\otimes\mathcal{O}_{S_{i}})\to0.
\end{equation} }
At each point $s\in S_{i},$ this gives an $(n-2)$-extended monad whose cohomology is an ideal of zero dimensional subscheme of $\mathbb{Y}.$ In particular, there exists a datum in $\mathcal{V}_{\mathbb{Y}}(c)^{st}$ representing it. Since $S_{i}$ is an open subset, one obtains a morphism $g_{\eta_{i}}:S_{i}\to\mathcal{V}_{\mathbb{Y}}(c)^{st},$ such that on the overlaps of the form $S_{i}\cap S_{j}$ one has 
$$g_{\eta_{i}}(S)\sim_{{\tiny GL(V_{1})}}g_{\eta_{j}}(s), \qquad\forall s\in S_{i}\cap S_{j}.$$ 
Thus giving rise to a well defined morphism $g_{\eta}:S\to\mathcal{M}_{\mathbb{Y}}(c)$ such that $\eta=(id_{\mathbb{Y}}\times g_{\eta})^{\ast}(\mathfrak{Z}).$ In particular, there exists a morphism $h_{\mathfrak{X}}:\Hilb^{[c]}(\mathbb{Y})\to\mathcal{M}_{\mathbb{Y}}(c),$ such that $\mathfrak{X}=(id_{\mathbb{Y}}\times h_{\mathfrak{X}})^{\ast}\mathfrak{Z}$ and $\mathcal{H}ilb_{\mathbb{Y}}(Spec(k))\cong\Hom(Spec(k),\mathcal{M}_{\mathbb{Y}}(c)),$ that is, $h=f^{-1}$ as a set-theoretic maps.

Moreover, for the inclusions $\tilde{\alpha}:\mathbb{Y}\times\mathcal{M}_{\mathbb{Y}}(c)\to\mathbb{C}^{n}\times
\Hilb^{[c]}(\mathbb{C}^{n})$ and $\tilde{\beta}:\mathbb{Y}\times\Hilb^{[c]}(\mathbb{Y})\to\mathbb{C}^{n}\times
\Hilb^{[c]}(\mathbb{C}^{n})$ one has $\tilde{\alpha}_{\ast}\mathfrak{Z}(U):=\mathfrak{Z}(U\cap\mathbb{C}^{n}\times
\Hilb^{[c]}({\mathbb{C}^{n}}))$ and $\tilde{\beta}_{\ast}\mathfrak{X}(U):=\mathfrak{X}(U\cap\mathbb{C}^{n}\times
\Hilb^{[c]}({\mathbb{C}^{n}})),$ on every open $U\subset\mathbb{C}^{n}\times\Hilb^{[c]}({\mathbb{C}^{n}}).$ Furthermore for a point $x=Spec(k)$ as in \eqref{points}, we have an isomorphism of stalks $$\lim_{V\ni x}[\tilde{\alpha}_{\ast}\mathfrak{Z}(V)]\cong\lim_{V\ni x}[\tilde{\beta}_{\ast}\mathfrak{X}(V)],$$ for all $V\subseteq U,$ in some directed system. Finally, the families $\mathfrak{Z}$ and $\mathfrak{X}$ are both restrictions, of the universal family $\mathcal{J}\to\mathbb{C}^{n}\times\Hilb^{[c]}({\mathbb{C}^{n}}),$ to subschemes, with the same topological support, and such that $\mathfrak{X}=(id_{\mathbb{Y}}\times h_{\mathfrak{X}})^{\ast}\mathfrak{Z}=[id_{\mathbb{Y}}\times (h_{\mathfrak{X}}\circ f)]^{\ast}\mathfrak{X}$ and $\mathfrak{Z}=(id_{\mathbb{Y}}\times f)^{\ast}\mathfrak{Z}=[id_{\mathbb{Y}}\times (f\circ h_{\mathfrak{X}})]^{\ast}\mathfrak{Z}.$ Hence $f$ is lifted to an (unique) isomorphism of schemes
$\Hilb^{[c]}(\mathbb{Y})\cong\mathcal{M}_{\mathbb{Y}}(c),$ with inverse $h_{\mathfrak{X}}$.

\bigskip

%%%%%%%%%%%%%%%%%%%%%%%%%%%%%%%%%%%%%%%%%%%%%%%%%%%%%%%%%%%%%%%%%%%%%%%%%%%%%%%%%%%%%%%%%%%%%%
%%%%%%%%%%%%%%%%%%%%%%%%%%%%%%%%%%%%%%%%%%%%%%%%%%%%%%%%%%%%%%%%%%%%%%%%%%%%%%%%%%%%%%%%%%%%%%

\section{Irreducible components of the Hilbert scheme of points}\label{irred}

The variety $\mathcal{C}(n,c)$ of $n$ commuting $c\times c$ matrices have been much studied by various authors since a 1961 paper by Gerstenhaber \cite{G}. The results concerning the irreducibility of $\mathcal{C}(n,c)$ can be summarized as follows:
\begin{itemize}
\item $\mathcal{C}(2,c)$ is irreducible for every $c$ (originally proved by Motzkin and Taussky \cite{MT}, see also \cite{G});
\item $\mathcal{C}(3,c)$ is irreducible for $c\le 10$ and reducible for $c\ge 29$, see \cite[Section 7.9]{OCV},
\cite{HO,S} and the references therein;
\item for $n\ge4$, $\mathcal{C}(n,c)$ is irreducible if and only if $c\le3$ \cite{G}.
\end{itemize}
In particular, determining the highest possible value of $c$ for which $\mathcal{C}(3,c)$ is irreducible is an important open problem.

On the other hand, much less is known about the irreducibility of the Hilbert scheme
$\Hilb^{[c]}({\mathbb{C}^{n}})$ of $c$ points on $\mathbb{C}^{n}$, see for instance
\cite[Section 7]{CEVV} and the references therein. 
\begin{itemize}
\item $\Hilb^{[c]}({\mathbb{C}^{2}})$ is irreducible for every $c$, see \cite{Fog};
\item $\Hilb^{[c]}({\mathbb{C}^{3}})$ is irreducible for $c\le 8$ \cite[Theorem 1.1]{CEVV}, while it is reducible for $c\ge 78$, cf. \cite{Iar85};
\item for $n\ge 4$, $\Hilb^{[c]}({\mathbb{C}^{n}})$ is irreducible if and only if for $c\le 7$
\cite[Theorem 1.1]{CEVV}.
\end{itemize}
Similarly, determining the highest possible value of $c$ for which $\Hilb^{[c]}({\mathbb{C}^{3}})$ is irreducible is also an important open question.

In this section, we connect the two problems through the following result.

\begin{proposition}\label{number}
The number of irreducible components of $\Hilb^{[c]}({\mathbb{C}^{n}})$ is smaller than, or equal to, the number of irreducible components of $\mathcal{C}(n,c)$. In particular, if $\mathcal{C}(n,c)$ is irreducible, then $\Hilb^{[c]}({\mathbb{C}^{n}})$ is also irreducible.
\end{proposition}

\begin{proof}
Clearly, the number of irreducible components of $\mathcal{C}(n,c)$ is the same as the number of irreducible components of $\mathcal{V}(n,c):=\mathcal{C}(n,c)\times\Hom(W,V)$. Let
$\mathcal{V}_1(n,c),\dots,$ $\mathcal{V}_p(n,c)$ denote the irreducible components of
$\mathcal{V}(n,c)$, and set \linebreak $\mathcal{V}_l(n,c)^{\rm st}:=\mathcal{V}_l(n,c)\cap\mathcal{V}(n,c)^{\rm st}$, with $l=1,\dots,p$.

It is possible that some components of $\mathcal{V}(n,c)$ contain no stable points; one can then order the irreducible components of $\mathcal{V}(n,c)$ in such a way that 
$\mathcal{V}_l(n,c)^{\rm st}\ne\emptyset$ for $l=1,\dots,q$ and $\mathcal{V}_l(n,c)^{\rm st}=\emptyset$ for $l=q+1,\dots,p$.

Since the group $G:=GL(V)$ is irreducible, it is easy to see that if $x\in\mathcal{V}_l(n,c)^{\rm st}$ then its orbit $G\cdot x \subset \mathcal{V}_l(n,c)^{\rm st}$. Note also that
$$ \mathcal{V}_l(n,c) /\!/_{\chi} G = \mathcal{V}_l(n,c)^{\rm st}/G $$
is irreducible, for each $l=1,\dots,q$. 

Since the GIT quotient $\mathcal{M}(n,c)$ coincides, by Proposition \ref{closed orbit}, with the set of stable $G$-orbits, we have that 
$$ \mathcal{M}(n,c) = 
\left( \mathcal{V}_1(n,c)^{\rm st}/G \right) \cup \cdots \cup 
\left( \mathcal{V}_q(n,c)^{\rm st}/G \right) $$
and the desired conclusion follows from Corollary \ref{identified}.
\end{proof}

As an immediate consequence of \cite[Theorems 26 \& 32]{S}, we obtain the following new result on the irreducibility of the Hilbert scheme of points in dimension $3$.

\begin{corollary} 
$\Hilb^{[c]}({\mathbb{C}^{3}})$ is irreducible for $c\le 10$.
\end{corollary}

As a final comment, we remark that determining which components of $\mathcal{V}(n,c)$ admit stable solutions seems to be a very interesting problem both from the linear algebra and the algebraic geometry points of view. More precisely, given an $n$-tuple of commuting matrices $(B_1,\dots,B_n)$, when is it possible to find a deformation $(B_1',\dots,B_n')$ and a vector $I$ such that the datum $(B_1',\dots,B_n',I)$ is stable?

%----------------------------------------------------------------------------------
%----------------------------------------------------------------------------------

\end{document}